\documentclass[10pt, letterpaper]{article}   
\usepackage{hyperref} 
\usepackage{amssymb} 
\usepackage{amsmath} 
\usepackage{mathrsfs} 
\usepackage{fullpage}
\usepackage{setspace}
\usepackage{graphicx}
\usepackage{enumerate}
\usepackage{amsthm}
\usepackage{MnSymbol}
\usepackage{mathtools}

\usepackage{tikz}
\tikzset{node distance=2cm, auto}
\usepackage{tikz-cd}
 
\usepackage{fancyhdr}
\newcommand\Q{\mathbb Q}
\newcommand\Z{\mathbb Z}
\newcommand\N{\mathbb N}
\newcommand{\Mod}[1]{\prescript{}{#1}{\mathrm{mod}}}

\newtheorem{theorem}{Theorem}[section]
\newtheorem{lemma}[theorem]{Lemma}
\newtheorem{proposition}[theorem]{Proposition}

\begin{document}

\begin{center}

\textbf{Unit Groups of Representation Rings and their Ghost Rings as Inflation Functors}\\

\vspace{.5cm}

Rob Carman\\
Department of Mathematics\\
University of California, Santa Cruz\\
wcarman@ucsc.edu\\

\end{center}

\vspace{.5cm}

\textbf{Abstract.} The theory of biset functors developed by Serge Bouc has been instrumental in the study of the unit group of the Burnside ring of a finite group, in particular for the case of $p$-groups.
The ghost ring of the Burnside ring defines an inflation functor, and becomes a useful tool in studying the Burnside ring functor itself.
We are interested in studying the unit group of another representation ring: the trivial source ring of a finite group.
In this article, we show how the unit groups of the trivial source ring and its associated ghost ring define inflation functors.
Since the trivial source ring is often seen as connecting the Burnside ring to the character ring and Brauer character ring of a finite group, we study all these representation rings at the same time.
We point out that restricting all of these representation rings' unit groups to their torsion subgroups also give inflation functors, which we can completely determine in the case of the character ring and Brauer character ring.

\section{Introduction}\label{Intro}
The long term goal of the author is to determine the units of finite order
of the trivial source ring of a finite group.
These units --
sometimes alternatively referred to as ``orthogonal units'' for reasons we will describe later -- 
give rise to certain autoequivalences of blocks
studied by Robert Boltje and Philipp Perepelitsky in \cite{BolPer}.
For $p$-groups, the trivial source ring is isomorphic to the Burnside ring,
and for $p'$-groups, the trivial source ring is isomorphic to
both the character and Brauer character rings.
So it makes sense to study all these representation rings together.
The unit group of the Burnside ring of a $p$-group has been determined by Serge Bouc already
in \cite{Bouc3} using his theory of biset functors.
And the finite order units of the character ring were completely determined
by Kenichi Yamauchi in \cite{Yamauchi}.
At the end of this article we use his result to completely determine the finite order units of the Brauer character ring.

In Section \ref{RepRings},
we describe all the representations rings that we will use throughout this article for a finite group $G$:
the Burnside ring, the trivial source ring, the character ring, and the Brauer character ring.
Each of the representation rings we consider has a distinguished automorphism of interest
whose square is the identity morphism.
Since these automorphisms are induced by taking dual modules,
any map between rings that respects the two rings' distinguished automorphisms
will be said to ``preserve duals.''

In section \ref{GhostRings},
we construct a ghost ring and ghost map for each representation ring.
In each case, the ghost map embeds the representation ring into its associated ghost ring,
which is free abelian with the same rank as its associated representation ring.
The ghost map will have a finite cokernel in every case.
The maps between each of the representation rings
has a unique extension to the level of ghost rings,
and we describe all of these extensions in this section as well.
We also point out that ghost rings have a duality operator,
and the relevant maps between ghost rings also preserve duals.

We recall the theory of algebraic maps developed by Andreas Dress in section \ref{AlgMaps}.
And in section \ref{WreathProd}, we recall the construction of the wreath product.

In section \ref{TenSet}, we consider two finite groups, $G$ and $H$,
and right-free $(G,H)$-bisets.
Given such a right-free biset $U$, we construct a tensor induction functor
$s_U:\prescript{}{H}{\mathrm{set}}\longrightarrow\prescript{}{G}{\mathrm{set}}$.
This construction is similar to one given by Bouc in \cite{Bouc},
and we list several of its interesting properties.
We then show that the functor $s_U$
induces a multiplicative map $B(U)$ between the Burnside rings of $H$ and $G$.
Readers familiar with the literature might notice
that this is also the notation used for the additive Burnside functor,
but since we only consider the multiplicative theory in this article,
there should be no confusion between the two.
In fact, when $U$ consists of a single $H$-orbit, the two coincide.

We describe the analogous tensor induction functor
for modules over group rings in section \ref{TenMod}.
Again this is naturally isomorphic to a tensor induction functor defined by Bouc in \cite{Bouc}.
Then in section \ref{TenFun},
we use this functor to define several multiplicative maps
between the other representation rings.
Also we describe how to extend all of these maps to the level of ghost rings
and show that these extensions are unique in a certain sense.
The correct extension was previously described in the case of the character ring
by David Gluck and Marty Isaacs in \cite{GluIsa}, at least in a particular case.
The extension for the trivial source ring was previously unknown however,
and we spend the bulk of this article discussing the correct extension,
with its relevant properties given in Theorem \ref{tUext}.

In section \ref{InfFun}, we quickly recall the notion of an inflation functor,
and in section \ref{UnitGp} we show that unit groups of each representation ring
(and its associated ghost ring) form inflation functors.
Lastly, we explain orthogonal units of representation rings and ghost rings in section \ref{OrthUnits},
and show that restricting to these subgroups of the respective unit groups
gives inflation functors as well.
We conclude this article with a complete description of the orthogonal unit group of the Brauer character ring in Theorem \ref{OrthBra}.

{\bf Notation and Conventions.}
Throughout this article, every $G$-set is assumed to be finite,
and we denote the category of $G$-sets by $\prescript{}{G}{\mathrm{set}}$.
And for any ring $R$, every $R$-module considered is assumed to be finitely generated.
We will denote the category of such modules by $\Mod{R}$.
When $M$ and $N$ are two $RG$-modules,
we use the notation $M\mid N$ to denote that $M$ is isomorphic to a direct summand of $N$.
When we say that $M$ is a trivial source $RG$-module,
we mean that $M\mid RX$ for some permutation $RG$-module $RX$ defined by a $G$-set $X$.
By a $(G,H)$-biset, we mean a finite set $U$ that is a left $G$-set and a right $H$-set
such that $g(uh)=(gu)h$ for all $g\in G, u\in U,$ and $h\in H$.
We say $U$ is right-free if $uh=uh'$ implies $h=h'$ always.
If $u\in U$ and $T$ is a subgroup of $H$ (denoted by $T\leq H)$,
we define $\prescript{u}{}{T}=\{g\in G:gu=ut\text{ for some }t\in T\}$,
which is a subgroup of $G$.
And if $S\leq G$, we define $S^u=\{h\in H:su=uh\text{ for some }s\in S\}$,
which is a subgroup of $H$.
When $U$ is right-free and $u\in U$, its stabilizer (by considering $U$ as a $G\times H$-set)
is the graph of the homomorphism
$\varphi_u:\prescript{u}{}{H}\longrightarrow H$ defined by $\varphi_u(g)=h$ if $gu=uh$.
In this situation, we see that $S^u=\varphi_u(S\cap\prescript{u}{}{H})$.
When $S\leq G,$ and $T\leq H$,
then also restricting actions on both sides makes $U$ an $(S,T)$-biset,
and by $S\backslash U/T$ we mean a set of representatives of the orbits of $U$
under these restricted actions.

{\bf Acknowledgments.} This material is based upon work supported by a grant from the University of California Institute for Mexico and the United States (UC MEXUS) and the Consejo Nacional de Ciencia y Tecnolog\'ia de M\'exico (CONACYT).
The author would like to thank Robert Boltje for his constant guidance and support throughout the duration of this project,
and to Serge Bouc for bringing the article \cite{Yamauchi} to his attention.

\section{Representation Rings}\label{RepRings}

We first briefly describe all the rings
and the morphisms between them
that we will consider throughout this article.
Let us fix a finite group $G$, a prime $p$,
and a $p$-modular system $(K,\mathcal O,F)$ large enough for $G$.
Assume that $(\pi)$ is the maximal ideal of $\mathcal O$, and $F=\mathcal O/(\pi)$.
We let $B(G)$ denote the Burnside ring of $G$,
$R_K(G)$ the character ring of $G$,
and $R_F(G)$ the Brauer character ring of $G$.
Also let $T_{\mathcal O}(G)$ denote the Grothendieck ring of the category of finitely generated trivial source $\mathcal OG$-modules.
Similarly, define $T_F(G)$ for trivial source $FG$-modules.
The functor $F\otimes_{\mathcal O}-:\Mod{\mathcal OG}\longrightarrow\Mod{FG}$
induces the canonical isomorphism
$T_{\mathcal O}(G)\stackrel{\sim}{\longrightarrow} T_F(G)$.
So we may identify $T_{\mathcal O}(G)$ with $T_F(G)$,
which we refer to as the trivial source ring of $G$.

Now if $X$ is a finite $G$-set, then the linearization
$\mathcal OX$ is a trivial source $\mathcal OG$-module,
and $FX$ is a trivial source $FG$-module.
So we have ring morphisms $B(G)\longrightarrow T_{\mathcal O}(G)$
and $B(G)\longrightarrow T_F(G)$.
It is easy to see that they commute with the isomorphism
$T_{\mathcal O}(G)\stackrel{\sim}{\longrightarrow} T_F(G)$,
so we will denote both maps by $l_G$.
If $G$ is a $p$-group, then the maps $l_G$ are isomorphisms.
If $V$ is a trivial source $\mathcal OG$-module,
then $K\otimes_{\mathcal O}V$ is a $KG$-module, hence has a $K$-character $\chi_V$.
We then have a ring morphism $c_G:T_{\mathcal O}(G)\longrightarrow R_K(G)$
which maps the class $[V]$ to $\chi_V$.
If $M$ is a trivial source $FG$-module, it has a Brauer character $\tau_M$,
and so we have a ring morphism $T_F(G)\longrightarrow R_F(G)$,
which maps the class $[M]$ to the Brauer character $\tau_M$.
Finally, we have the decomposition map $d_G:R_K(G)\longrightarrow R_F(G)$.
In the case where $p$ doesn't divide the order of $G$,
all three of $b_G,c_G,$ and $d_G$ are isomorphisms. 	
Altogether we have the following commutative diagram:

\begin{equation}
\label{reprings}
\raisebox{-0.5\height}{\includegraphics{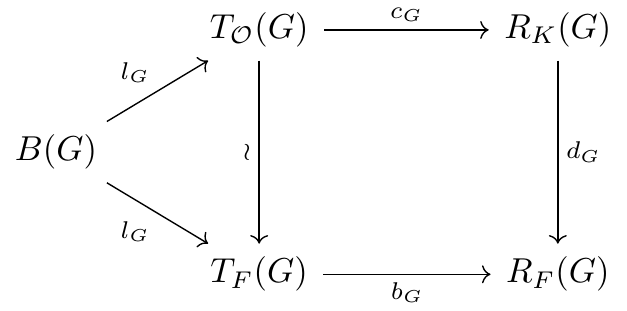}}
\end{equation}

These are the representation rings and maps we will focus on throughout.
Recall that each of these representation rings are free abelian with finite rank.
See \cite{CurRein} and \cite{Schneider} for more details.
Notice also that each of these representation rings (except for $B(G)$)
has a distinguished ring automorphism induced by taking dual modules.
Moreover, the square of such an automorphism is always the identity morphism.
For instance, if $M$ is a trivial source $FG$-module,
then its dual module $M^\circ=\text{Hom}_F(M,F)$ is again a trivial source $FG$-module.
So if $a\in T_F(G)$, then $a=[M]-[N]$ for some trivial source $FG$-modules $M$ and $N$.
The dual of $a$ is $a^\circ:=[M^\circ]-[N^\circ]$.
Notice $(a^\circ)^\circ=a$ for all $a\in T_F(G)$.
Similarly, we have a dual operator on $T_{\mathcal O}(G), R_F(G),$ and $R_K(G)$.
For completion, we just consider the identity operator on $B(G)$
as the distinguished automorphism of interest.
We denote all these automorphisms by $-^\circ$.
Since any permutation module is isomorphic to its dual module,
we see that the maps $l_G$ preserve duals.
It is easy to see that all the other ring morphisms preserve duals.
That is, for instance, $b_G(a^\circ)=b_G(a)^\circ$ for all $a\in T_F(G)$.

\section{Ghost Rings}\label{GhostRings}
Now for each of the representation rings in the previous section
we want to describe an associated ghost ring and ghost map.
First, if we let $\mathscr{S}(G)$ denote the set of subgroups of $G$, we have the ring morphism

$$\phi_G:B(G)\longrightarrow\prod_{S\in\mathscr{S}(G)}\Z,\quad[X]\mapsto(|X^S|)_{S\leq G},$$

\noindent where $X^S$ denotes the subset of $X$ fixed by $S$.
It is well-known that this map is injective and that its image lies in the $G$-fixed points
when considering the conjugation action of $G$ on $\prod_{S\in\mathscr{S}(G)}\Z$.
So if we set $\tilde B(G):=(\prod_{S\in\mathscr{S}(G)}\Z)^G,$
then $\phi_G$ is a ring morphism $B(G)\longrightarrow\tilde B(G)$,
which is well-known to have finite cokernel.
We refer to $\tilde B(G)$ as the ghost ring of $B(G)$,
and $\phi_G$ as the associated ghost map.

Next we let $e\in\N$ be the exponent of $G$.
We can then write $e=p^ah$ for some $a,h\in\N$ with $p\nmid h$.
Let $\zeta\in K$ be a primitive $e$th root of unity.
In fact, $\zeta\in\mathcal O$.
Hence $\Z[\zeta]\subseteq\mathcal O$ and $\Q(\zeta)\subseteq K$.
We set $\Gamma:=\text{Gal}(\Q(\zeta)/\Q)$.
Every element of $\Gamma$ is of the form $\gamma_i$
where $i$ is an integer relatively prime to $e$ and $\gamma_i(\zeta)=\zeta^i$.
For each such $\gamma_i\in\Gamma,$ we let $i^*$ be an integer such that $ii^*\equiv1\pmod e$.
Hence $\gamma_{i^*}=\gamma_i^{-1}$.
Now if we let $\mathscr{E}(G)$ denote the set of elements of $G$,
we have a ring morphism

$$\varepsilon_G:R_K(G)\longrightarrow\prod_{x\in\mathscr{E}(G)}\Z[\zeta],
\quad\chi\mapsto(\chi(x))_{x\in\mathscr{E}(G)},$$

\noindent where $\chi$ denotes a (virtual) character of $G$.
Now $G$ acts on the ring $\prod_{x\in\mathscr{E}(G)}\Z[\zeta]$ via conjugation,
and also $\Gamma$ acts on this ring via the action
$\gamma_i\cdot(w_x)_{x\in\mathscr{E}(G)}=(\gamma_i(w_{x^{i^*}}))_{x\in\mathscr{E}(G)}$.
These two actions clearly commute,
hence $G\times\Gamma$ acts on the codomain of $\varepsilon_G$.
The image of $\varepsilon_G$ is fixed by $G\times\Gamma$,
so if we let $\tilde R_K(G)$ be the fixed point subring
$(\prod_{x\in\mathscr{E}(G)}\Z[\zeta])^{G\times\Gamma}$,
we see that we have a ring morphism $\varepsilon_G:R_K(G)\longrightarrow\tilde R_K(G)$.
It is well-known that $\varepsilon_G$ is injective with finite cokernel.
So $\tilde R_K(G)$ and $\varepsilon_G$ will be the ghost ring and ghost map of $R_K(G)$.
Similarly, if we set $\mu:=\zeta^{p^a}$,
then $\mu$ is a primitive $h$th root of unity in $\mathcal O$,
where $h$ is the $p'$-part of the exponent of $G$.
So if $\tau\in R_F(G)$ is a (virtual) Brauer character of $G$,
then $\tau$ is a function on the set of $p$-regular elements of $G$,
taking values in $\Z[\mu]\subseteq\mathcal O$.
If we let $\mathscr{E}_p(G)$ denote the set of $p$-regular elements of $G$,
then we have a ring morphism

$$\xi_G:R_F(G)\longrightarrow\prod_{y\in\mathscr{E}_p(G)}\Z[\mu],
\quad\tau\mapsto(\tau(y))_{y\in\mathscr{E}_p(G)}.$$

\noindent Again, if we set $\Delta:=\text{Gal}(\Q(\mu)/\Q)$,
every element of $\Delta$ is of the form $\delta_i$,
where $i\in\Z$ is relatively prime to $h$, and $\delta_i(\mu)=\mu^i$.
Then similar to above, $G\times\Delta$ acts on $\prod_{y\in\mathscr{E}_p(G)}\Z[\mu]$.
So we let $\tilde R_F(G):=(\prod_{y\in\mathscr{E}_p(G)}\Z[\mu])^{G\times\Delta}$
be the ghost ring of $R_F(G),$
and $\xi_G:R_F(G)\longrightarrow\tilde R_F(G)$ is the associated ghost map.

We next want to define a ghost ring for $T_{\mathcal O}(G)$ and $T_F(G)$
and ghost maps that commute with the canonical isomorphism
$T_{\mathcal O}(G)\stackrel{\sim}{\rightarrow}T_F(G)$.
Notice first that if $M$ is a trivial source $\mathcal OG$- or $FG$-module,
and $P\leq G$ is a $p$-subgroup,
then the Brauer construction $M(P)$ is an $F[N_G(P)/P]$-module,
hence has a Brauer character that takes values in $\Z[\mu]\subseteq\mathcal O$.
Let us define the set $\mathscr{T}_p(G)$
to be the set of all pairs $(E,c)$ where $E$ is a $p$-hypo-elementary subgroup of $G$
with $\langle c\rangle=E/O_p(E),$ a cyclic $p'$-group.
For a pair $(E,c)\in\mathscr{T}_p(G)$ and a trivial source module $M$,
if we let $\tau_{M,E}$ denote the Brauer character of $M(O_p(E))$,
we can define

$$\tau_G:T_F(G)\longrightarrow\prod_{(E,c)\in\mathscr{T}_p(G)}\Z[\mu],
\quad[M]\mapsto(\tau_{M,E}(c))_{(E,c)\in\mathscr{T}_p(G)}.$$
Then $G\times\Delta$ acts on the codomain of this function via
$(x,\delta_i)\cdot(z_{(E,c)})_{(E,c)\in\mathscr{T}_p(G)}=
(\delta_i(z_{(E^x,(c^{i^*})^x)})_{(E,c)\in\mathscr{T}_p(G)},$
and the image of $\tau_G$ is fixed by this action.
So if we set $\tilde T_F(G):=(\prod_{(E,c)\in\mathscr{T}_p(G)}\Z[\mu])^{G\times\Delta}$,
then $\tilde T_F(G)$ will be the ghost ring of $T_F(G)$
and $\tau_G:T_F(G)\longrightarrow\tilde T_F(G)$ is the associated ghost map.
Similarly, we can set $\tilde T_{\mathcal O}(G):=\tilde T_F(G)$,
and we have a morphism $T_{\mathcal O}(G)\longrightarrow \tilde T_{\mathcal O}(G)$,
which we will also denote by $\tau_G$.
This makes sense to do
since applying the Brauer construction to a trivial source $\mathcal OG$-module
is the same as applying the functor $F\otimes_{\mathcal O}-$
and then applying the Brauer construction to the resulting trivial source $FG$-module.
In other words, the isomorphism $T_{\mathcal O}(G)\stackrel{\sim}{\rightarrow}T_F(G)$
extends via the ghost maps to the equality $\tilde T_{\mathcal O}(G)=\tilde T_F(G)$.

We similarly want to extend the maps $l_G, b_G, c_G,$ and $d_G$
to the level of ghost rings via the various ghost maps.
Notice that such extensions must be unique since
each ghost ring contains its associated representation ring as a finite index subgroup,
and all the ghost maps are additive.
First we explain an extension of $l_G$.
If $(E,c)\in\mathscr{T}_p(G)$, then in particular $E$ is a subgroup of $G$.
So we can define the function

$$\tilde l_G:\tilde B(G)\longrightarrow\tilde T_{\mathcal O}(G),
\quad(n_S)_{S\in\mathscr{S}(G)}\mapsto(n_E)_{(E,c)\in\mathscr{T}_p(G)},$$

\noindent which is clearly a ring morphism.
And if $X$ is a $G$-set,
then $(\tau_G\circ l_G)([X])=(\tau_{FX,E}(c))_{(E,c)\in\mathscr{T}_p(G)}$.
Now $(FX)(O_p(E))\cong FX^{O_p(E)}$ as permutation $F[N_G(O_p(E))/O_p(E)]$-modules,
so the value of the Brauer character of $(FX)(O_p(E))$ at $c$ is the number of fixed points:
$|(X^{O_p(E)})^{\langle c\rangle}|=|X^{\langle O_p(E),c\rangle}|=|X^E|.$
On the other hand,
$(\tilde l_G\circ\phi_G)([X])=\tilde l_G((|X^S|)_{S\in\mathscr{S}(G)}
=(|X^E|)_{(E,c)\in\mathscr{T}_p(G)}$.
So we see $\tau_G\circ l_G=\tilde l_G\circ\phi_G$,
that is $\tilde l_G$ extends $l_G$ to the appropriate ghost rings.

Next we describe an extension of $b_G$.
If $y\in\mathscr{E}_p(G)$ and $\mathbf 1$ denotes the trivial subgroup of $G$,
then the subgroup $\langle y\rangle\leq G$ is $p$-hypo-elementary
with $O_p(\langle y\rangle)=\mathbf 1$.
Hence $(\langle y\rangle,\{y\})\in\mathscr{T}_p(G)$.
So we can define the function

$$\tilde b_G:\tilde T_F(G)\longrightarrow\tilde R_F(G),
\quad(z_{(E,c)})_{(E,c)\in\mathscr{T}_p(G)}\mapsto
(z_{(\langle y\rangle,\{y\})})_{y\in\mathscr{E}_p(G)},$$

\noindent which is clearly a ring morphism.
We see that if $M$ is a trivial source $FG$-module,
then $(\xi_G\circ b_G)([M])$ and $(\tilde b_G\circ\tau_G)([M])$
both take the value of the Brauer character of $M$ at $y$ for all $y\in\mathscr{E}_p(G)$.
Hence $\xi_G\circ b_G=\tilde b_G\circ\tau_G$,
and therefore $\tilde b_G$ is an extension of $b_G$ via the appropriate ghost rings.

Similarly, we extend the function $c_G$ in the following way:
If $x\in\mathscr{E}(G)$, we can write $x=x_px_{p'}$,
where $x_p$ is the $p$-part of $x$, and $x_{p'}$ is the $p'$-part of $x$.
Then $\langle x\rangle\leq G$ with $O_p(\langle x\rangle)=\langle x_p\rangle$
and $\langle x\rangle/O_p(\langle x\rangle)=
\langle x\langle x_p\rangle\rangle$.
Hence $\langle x\rangle$ is $p$-hypo-elementary
with $(\langle x\rangle,x\langle x_p\rangle)\in\mathscr{T}_p(G)$.
So we can define the function

$$\tilde c_G:\tilde T_{\mathcal O}(G)\longrightarrow\tilde R_K(G),
\quad(z_{(E,c)})_{(E,c)\in\mathscr{T}_p(G)}\mapsto
(z_{(\langle x\rangle,x\langle x_p\rangle)})_{x\in\mathscr{E}(G)},$$

\noindent which is clearly a ring morphism extending $c_G$.
Lastly, we describe the extension of $d_G$.
Since $\mathscr{E}_p(G)\subseteq\mathscr{E}(G),$ we can simply define the function

$$\tilde d_G:\tilde R_K(G)\longrightarrow\tilde R_F(G),
\quad(w_x)_{x\in\mathscr{E}(G)}\mapsto(w_y)_{y\in\mathscr{E}_p(G)}.$$

\noindent It is easy to see $\tilde d_G$ extends $d_G$
since two characters have equal image under $d_G$ if and only if
the characters take the same values on all of $\mathscr{E}_p(G)$.
Notice that $\tilde d_G$ is surjective, just like $d_G$.
Altogether we get the following commutative diagram of ghost rings:

\begin{equation}
\label{ghostrings}
\raisebox{-0.5\height}{\includegraphics{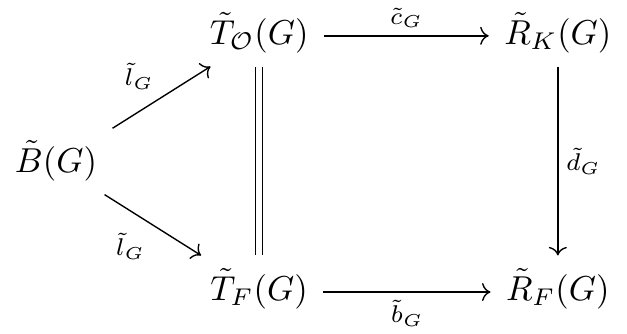}}
\end{equation}

So Diagram \ref{ghostrings} extends Diagram \ref{reprings} to the level of ghost rings.
As we noted all the representation rings in Diagram \ref{reprings} have a duality operator,
we also want to note all the ghosts rings also have a duality operator.
As with $B(G)$, we just consider the identity on $\tilde B(G)$.
Now since $-1$ is always relatively prime to $e$, we have $\gamma_{-1}\in\Gamma$,
and $\gamma_{-1}^2=\gamma_1$ is the identity of $\Gamma$.
So we see that $\gamma_{-1}$ induces an automorphism of $\tilde R_K(G)$
whose square is the identity morphism.
This defines the duality operator on $\tilde R_K(G)$.
To be more explicit, if $(w_x)_{x\in\mathscr{E}(G)}\in\tilde R_K(G)$,
then its dual element is
$(w_x)^\circ_{x\in\mathscr{E}(G)}=(\gamma_{-1}(w_x))_{x\in\mathscr{E}(G)}$.
We similarly can define duality operators on the other ghost rings
$\tilde T_{\mathcal O}(G), \tilde T_F(G),$ and $\tilde R_F(G)$ since also $\delta_{-1}\in\Delta$.
It is clear that all the morphisms in Diagram \ref{ghostrings} commute with
these various duality operators.

\section{Algebraic Maps}\label{AlgMaps}

Here we recall some of the theory of algebraic maps developed by Andreas Dress
which we will use throughout this article.
The setup is the following: Let $A$ be a semiring and $E$ a commutative ring.
We consider set maps $f:A\longrightarrow E$ and for an $a\in A,$ we define

$$D_af:A\longrightarrow E,\hspace{.1cm}x\mapsto f(x+a)-f(x).$$

We say that $f:A\longrightarrow E$ is \emph{algebraic} if there exists some $n\in\N$ such that

$$D_{a_1}D_{a_2}\cdots D_{a_{n+1}}f=0\hspace{.2cm}\text{for all }a_1,a_2,\dots,a_{n+1}\in A.$$

\noindent If such an $n\in\N$ exists,
then the least such $n$ will be called the \emph{degree} of $f$.
If $g:A\longrightarrow E$ is an additional map,
then we can define the pointwise addition $f+g:A\longrightarrow E$
by $(f+g)(x)=f(x)+g(x)$ for all $x\in A$.
Similarly, we can define the product $fg$ by $(fg)(x)=f(x)g(x)$.
And if $c\in E$ is a constant,
then we can consider the map $cf$ defined by $(cf)(x)=cf(x)$.
The set of all functions $A\longrightarrow E$ is then an $E$-module,
and we can therefore talk about linear independence over $E$
for collections of functions $A\longrightarrow E$.
At this point, we would like to collect a few facts about algebraic maps that we will use later.
The proofs can be found in \cite{CurRein} and \cite{Dress}.

\begin{proposition}\label{algprops}
Let $A$ be a semiring and let $E$ and $E'$ be commutative rings.
\begin{enumerate}
\item A nonconstant function $f:A\longrightarrow E$ is algebraic of degree 1
if and only if $f=h+c$, where $c\in E$ is constant and $h:A\longrightarrow E$
is a nonzero additive map.
\item If $f:A\longrightarrow E$ is algebraic of degree $n$ and $c\in E$ is constant,
then $cf$ is algebraic of degree $\leq n$,
with equality if $c$ is not a zero divisor in $E$.
\item If $f,g:A\longrightarrow E$ are algebraic of degrees $m$ and $n$,
then $f+g$ is algebraic of degree $\leq\max\{m,n\}$,
with equality if $f$ and $g$ are linearly independent over $E$.
\item If $f,g:A\longrightarrow E$ are algebraic of degrees $m$ and $n$,
then $fg$ is algebraic of degree $\leq m+n$.
\item If $f:A\longrightarrow E$ is algebraic of degree $m$
and $g:E\longrightarrow E'$ is algebraic of degree $n$,
then $g\circ f$ is algebraic of degree $\leq mn$.
\item Suppose $f:A\longrightarrow E$ and $a_0\in A$ are such that
$D_{a_0}f$ is algebraic of degree $n$ and $D_af$ is algebraic of degree $\leq n$ for all $a\in A$.
Then $f$ is algebraic of degree $n+1$.
\item If $f:A\longrightarrow E$ is algebraic of degree $n$ and $i:A\longrightarrow\bar A$
is the canonical map from $A$ into its associated Grothendieck ring $\bar A$,
then there exists a unique map $\bar f:\bar A\longrightarrow E$ such that $\bar f\circ i=f$.
Moreover, $\bar f$ is algebraic of degree $n$, and if $f$ is multiplicative, then so is $\bar f$.
\end{enumerate}
\end{proposition}

We will use the first six properties of the above proposition
to show that various maps are algebraic,
and we will use the last property to show
that if two algebraic maps are equal on an additive generating set of a ring,
then the uniqueness of the statement implies the maps must agree on the whole ring.
We will often refer to this fact as the Theorem of Dress.

\section{Wreath Product}\label{WreathProd}

Here we recall the important group theoretic construction of the wreath product.
We will use the following notation throughout the rest of the paper.
Let $H$ be a finite group, and let $n$ be some natural number.
The symmetric group $S_n$ acts on $H^n=H\times\cdots\times H$ ($n$ copies)
on the left by

$$\pi(h_1,\dots,h_n)=(h_{\pi^{-1}(1)},\dots,h_{\pi^{-1}(n)})
\hspace{.2cm}\text{ for all }\pi\in S_n\text{ and }h_1,\dots,h_n\in H.$$

\noindent So we can form the semidirect product $H^n\rtimes S_n,$
whose multiplication is given by 

$$((h_1,\dots,h_n),\pi)((k_1,\dots,k_n),\sigma)=((h_1,\dots,h_n)\pi(k_1,\dots,k_n),\pi\sigma)
=((h_1k_{\pi^{-1}(1)},\dots,h_nk_{\pi^{-1}(n)}),\pi\sigma).$$

We will use the notation $(h_1,\dots,h_n;\pi)$ for $((h_1,\dots,h_n),\pi)$,
and write $H\wr S_n:=H^n\rtimes S_n$.
If $1_n$ denotes the identity of $S_n$, and we denote the identity of $H$ by $1_H$,
then the identity of $H\wr S_n$ is $(1_H,\dots,1_H;1_n)$,
and inverses are given by
$(h_1,\dots,h_n;\pi)^{-1}=(h^{-1}_{\pi(1)},\dots,h^{-1}_{\pi(n)};\pi^{-1}).$
We see that $S_n$ and $H^n$ are both embedded in $H\wr S_n$ via the homomorphisms

$$\iota:S_n\longrightarrow H\wr S_n,\quad \pi\mapsto(1_H,\dots,1_H;\pi)
\quad\text{ and }\quad
\kappa:H^n\longrightarrow H\wr S_n,\quad (h_1,\dots,h_n)\mapsto(h_1,\dots,h_n;1_n).$$

\section{Tensor Induction for $H$-Sets}\label{TenSet}

Again let $H$ be a finite group and $n$ a natural number.
If $X$ is an $H$-set, then $X^n$ is an $H\wr S_n$-set
via $(h_1,\dots,h_n;\pi)\cdot(x_1,\dots,x_n)=(h_1x_{\pi^{-1}(1)},\dots,h_nx_{\pi^{-1}(n)})$.
Now if $Y$ is an additional $H$-set and $f:X\longrightarrow Y$ is an $H$-map,
we see that the function
$f^n:X^n\longrightarrow Y^n,\quad(x_1,\dots,x_n)\mapsto(f(x_1),\dots,f(x_n))$
is a morphism of $H\wr S_n$-sets.
This gives us a functor $-^n:\prescript{}{H}{\text{set}}\longrightarrow\prescript{}{H\wr S_n}{\text{set}}$.

Now let $G$ be an additional finite group,
and let $U$ be a finite right-free $(G,H)$-biset.
We pick $u_1,\dots,u_n\in U$ such that the ordered set $(u_1,\dots,u_n)$
is a complete set of representatives of the $H$-orbits of $U$.
So here $n=|U/H|$.
Now for $g\in G$, we have $gu_i=u_{\pi(i)}h_i$ for some $\pi\in S_n$ and $h_i\in H$.
Since $U$ is right-free, $\pi$ and the $h_i$ are uniquely determined by $g$.
We can therefore define a function

$$\theta:G\longrightarrow H\wr S_n,\quad
g\mapsto\iota(\pi)\kappa(h_1,\dots,h_n)=(h_{\pi^{-1}(1)},\dots,h_{\pi^{-1}(n)};\pi).$$

\noindent It is easy to check that $\theta$ is a homomorphism,
hence induces a restriction functor
$\text{Res}(\theta):\prescript{}{H\wr S_n}{\text{set}}\longrightarrow\prescript{}{G}{\text{set}}.$
We then define the functor
$s_U:\prescript{}{H}{\text{set}}\longrightarrow\prescript{}{G}{\text{set}}$
as the composition of $\text{Res}(\theta)$ and $-^n$.
So explicitly, if $X$ is an $H$-set, then $s_U(X)=X^n$,
and if $g\in G$ with $gu_i=u_{\pi(i)}h_i$, then

$$g\cdot(x_1,\dots,x_n)=(h_{\pi^{-1}(1)}x_{\pi^{-1}(1)},\dots,h_{\pi^{-1}(n)}x_{\pi^{-1}(n)})$$

\noindent for all $(x_1,\dots,x_n)\in X^n$.
Now of course the functor $s_U$ depends on the choice of $U/H$,
but it is easy to see that different choices of $U/H$ lead to naturally isomorphic functors.

Next we gather a few of the properties of the functor $s_U$.
In Section 4 of \cite{Bouc}, Serge Bouc gives a different - but naturally isomorphic -
definition of $s_U$.
His construction is slightly more general, however, since he does not require that $U$ be right-free.
This additional property will be necessary for later constructions,
so we require it here.
If $V\subseteq U$ such that $v\in V,h\in H$ implies $vh\in V$,
we say that $V$ is $H$-invariant, and denote this by $V\subseteq_HU$.
The action of $G$ on $U$ induces an action on the set of all $H$-invariant subsets of $U$,
and we denote the stabilizer of $V$ in $G$ by $G_V$.
Notice that if $V\subseteq_HU$, then also $U-V\subseteq_HU$, and $G_{U-V}=G_V$,
so both $V$ and $U-V$ are right-free $(G_V,H)$-bisets.
Also, if $V$ is a right-free $(H,K)$-biset for some other finite group $K$,
then $U\times V$ is a right $H$-set via $(u,v)h=(uh,h^{-1}v)$.
We denote by $U\times_HV$ the set of $H$-orbits of $U\times V$ under this action.
Such an orbit will be denoted by $(u,_Hv)$ for $u\in U$ and $v\in V$.
Then $U\times_HV$ is naturally a right-free $(G,K$)-biset via $g(u,_Hv)k=(gu,_Hvk)$.
Proofs of the properties to follow are given in Bouc's paper, so we omit them here.\\

\begin{proposition}\label{sUprops}
Let $U,U'$ be right-free $(G,H)$-bisets, let $V$ be a right-free $(K,G)$-biset,
and let $X,Y$ be $H$-sets.
Then the following hold:
\begin{enumerate}
\item $s_U(\bullet)\cong\bullet$ as $G$-sets, where $\bullet$ denotes a one-point set.
\item $s_U(X)\cong s_{U'}(X)$ as $G$-sets whenever $U\cong U'$ as $(G,H)$-bisets,
where the isomorphism is natural in $X$.
\item $s_U(X\times Y)\cong s_U(X)\times s_U(Y)$ as $G$-sets,
where the isomorphism is natural in both $X$ and $Y$.
\item $s_{U\sqcup U'}(X)\cong s_U(X)\times s_{U'}(X)$ as $G$-sets,
where the isomorphism is natural in $X$.
\item $s_{V\times_GU}(X)\cong(s_V\circ s_U)(X)$ as $K$-sets,
where the isomorphism is natural in $X$.
\item $s_U(X\sqcup Y)\cong\displaystyle\coprod_{V\in G\backslash\{V\subseteq_HU\}}\mathrm{Ind}_{G_V}^G(s_V(X)\times s_{U-V}(Y))$ as $G$-sets,
where the isomorphism is natural in both $X$ and $Y$.
\end{enumerate}
\end{proposition}

Now if we let $B^+(H)\subseteq B(H)$
denote the semiring generated by the isomorphism classes of $H$-sets,
then $s_U$ induces a function $B(U):B^+(H)\longrightarrow B(G)$
that sends $[X],$ the isomorphism class of an $H$-set $X$,
to $[s_U(X)],$ the isomorphism class of $s_U(X)$.
By the functoriality of $s_U$ this function is well-defined.
And since different choices of $U/H$ give naturally isomorphic functors,
this function does not depend on such a choice.
Now property 3 of Proposition \ref{sUprops} shows that this function is multiplicative on $B^+(H)$.
We will show this function is algebraic of degree $n=|U/H|$.\\

\begin{lemma}\label{sUalgebraic}
Let $U$ be a right-free $(G,H)$-biset.
The function $B(U):B^+(H)\longrightarrow B(G), [X]\mapsto[s_U(X)]$ is algebraic of degree $|U/H|$.
\end{lemma}
\begin{proof}
We prove this lemma by induction on $n=|U/H|$.
So first suppose that $n=1$.
Then fix $u\in U$.
For $g\in G$, we have $gu=uh$ for some unique $h\in H$.
Then $\theta:G\longrightarrow H=H\wr S_1, g\mapsto h$ is the homomorphism defining $s_U$.
And since $n=1,$ the functor $s_U$ is just the restriction functor along $\theta$.
Hence $B(U)$ is just the restriction map.
And since restriction is additive, we see that $B(U)$ is algebraic of degree $\leq1$.
But $B(U)$ is nonconstant, so it must be of degree $1=n$.

Now assume for some fixed $k\in\N$,
that we have shown that if $V$ is a right-free biset with $|V/H|\leq k$,
then $B(V)$ is algebraic of degree $|V/H|$.
Then suppose that $U$ is a right-free $(G,H)$-biset with $n=|U/H|=k+1$.
Let us fix $[X],[Y]\in B^+(H)$.
We then have the following:

\begin{align*}
D_{[Y]}B(U)([X])
&=B(U)([X]+[Y])-B(U)([X])\\
&=B(U)([X\sqcup Y])-B(U)([X])\\
&=[s_U(X\sqcup Y)]-[s_U(X)]\\
&=[\coprod_{V\in G\backslash\{V\subseteq_HU\}}\text{Ind}_{G_V}^G(s_V(X)\times s_{U-V}(Y))]-[s_U(X)]\\
&=\sum_{V\in G\backslash\{V\subseteq_HU\}}[\text{Ind}_{G_V}^G(s_V(X)\times s_{U-V}(Y))]-[s_U(X)]\\
&=[\text{Ind}_{G_U}^G(s_U(X)\times s_{U-U}(Y))]+\sum_{V\in G\backslash\{V\subsetneq_HU\}}[\text{Ind}_{G_V}^G(s_V(X)\times s_{U-V}(Y))]-[s_U(X)]\\
&=[s_U(X)]-[s_U(X)]+\sum_{V\in G\backslash\{V\subsetneq_HU\}}[\text{Ind}_{G_V}^G(s_V(X)\times s_{U-V}(Y))]\\
&=\sum_{V\in G\backslash\{V\subsetneq_HU\}}\text{Ind}_{G_V}^G(B(V)([X])B(U-V)([Y]))
\end{align*}

Now for every $V\subseteq_HU$ such that $V\neq U,$ we must have $|V/H|\leq k.$
So by the induction hypothesis, $B(V)$ is algebraic of degree $|V/H|$.
Now $B(U-V)([Y])$ is constant in terms of $[X]$,
so the function $B(V)B(U-V)([Y])$ is algebraic of degree $\leq|V/H|$.
And since induction is additive (algebraic of degree 1),
composing with $\text{Ind}_{G_V}^G$ again gives an algebraic function of degree $\leq|V/H|$.
So $D_{[Y]}B(U)$ is a sum of functions of degrees $\leq k$ for all $[Y]\in B^+(H)$.
When $Y\cong\bullet,$ a one-point $H$-set, we see that
$D_{\bullet}B(U)([X])=\sum_{V\in G\backslash\{V\subsetneq_HU\}}\text{Ind}_{G_V}^G(B(V)([X])).$
After choosing a set of representatives of $G\backslash\{V\subseteq_HU\}$,
it is clear that the functions $\textrm{Ind}_{G_V}^G\circ B(V)$ are linearly independent over $B(G)$.
And for any maximal $V\subsetneq_HU$,
the function $\textrm{Ind}_{G_V}^G\circ B(V)$ is algebraic of degree $k$
by the induction hypothesis.
Hence $D_{\bullet}B(U)$ is algebraic of degree $k$ by Property 3 of Proposition \ref{algprops}.
Finally $B(U)$ is algebraic of degree $k+1=n$ by Property 6 of the proposition.
\end{proof}

Now by applying the theorem of Dress,
we see that $B(U)$ extends to a function $B(H)\longrightarrow B(G)$
that is still multiplicative and algebraic of degree $|U/H|$.
We will again denote this function by $B(U)$ and refer to it as \emph{tensor induction} by $U$.
Mostly we are interested in the case where $H$ is a subgroup of $G$
and $U$ is just $G$ as a right-free $(G,H)$-biset with multiplication from $G$ on the left
and multiplication from $H$ on the right.
Then the map $B(H)\longrightarrow B(G)$ is the multiplicative induction map,
which is algebraic of degree $[G:H]$.

Before moving on, we first recall (see Section 3b of \cite{Yoshida})
how to extend $B(U)$ to a multiplicative function between ghost rings.
Recall that if $S\leq G$ and $u\in U$, then $S^u=\varphi_u(S\cap\prescript{u}{}{H})\leq H$.
So we can define the function

$$\tilde B(U):\tilde{B}(H)\longrightarrow\tilde{B}(G),\quad
(n_T)_{T\in\mathscr{S}(H)}
\mapsto\left(\prod_{u\in S\backslash U/H}n_{S^u}\right)_{S\in\mathscr{S}(G)}.$$

\noindent We see clearly that $\tilde B(U)$ is multiplicative.
We can also see that $\tilde B(U)$ is algebraic of degree
$\max\{|S\backslash U/H|:S\in\mathscr{S}(G)\}=|\{1_G\}\backslash U/H|=|U/H|$, as is $B(U)$.
We also have $\tilde B(U)\circ\phi_H=\phi_G\circ B(U)$,
and $\tilde B(U)$ is the unique multiplicative function with this property.
In other words, $\tilde B(U)$ is the multiplicative extension of $B(U)$ to ghost rings.

\section{Tensor Induction of Modules}\label{TenMod}

For this section, we let $R$ denote any commutative ring.
We call an $RH$-module $M$ a \emph{representation module} if $M$ is finitely generated
and $M$ is free as an $R$-module.
Just as we studied the tensor induction of $H$-sets in the previous section,
we here study the tensor induction of representation $RH$-modules.
One construction is defined by Bouc in \cite{Bouc}.
Here we give a different construction that still results in isomorphic modules.
Now let $n\in\N$.
If $M$ is a representation $RH$-module,
then we see that $M^{\otimes n}:=M\otimes_R\cdots\otimes_RM$ ($n$ copies)
is a representation $R[H\wr S_n]$-module with action given by

$$(h_1,\dots,h_n;\pi)\cdot(m_1\otimes\cdots\otimes m_n)
=(h_1m_{\pi^{-1}(1)})\otimes\cdots\otimes(h_nm_{\pi^{-1}(n)}).$$

This then defines a functor $-^{\otimes n}$ from the category of representation $RH$-modules
to the category of representation $R[H\wr S_n]$-modules.
Now if $U$ is a right-free $(G,H)$-biset,
then after picking $(u_1,\dots,u_n)$, an ordered set of representatives of $U/H$,
we get a group homomorphism $\theta:G\longrightarrow H\wr S_n$ as before.
Then $\theta$ induces a restriction functor, and after composing with $-^{\otimes n}$,
we get a functor $t_U$ from the category of representation $RH$-modules
to the category of representation $RG$-modules.
Explicitly, $t_U(M)=M^{\otimes n}=M\otimes_R\cdots\otimes_RM$,
and if $g\in G$ with $gu_i=u_{\pi(i)}h_i$ for some $\pi\in S_n, h_i\in H$, then

$$g\cdot(m_1\otimes\cdots\otimes m_n)
=(h_{\pi^{-1}(1)}m_{\pi^{-1}(1)})\otimes\cdots\otimes(h_{\pi^{-1}(n)}m_{\pi^{-1}(n)}).$$

Of course $t_U$ depends on the choice of representatives of $U/H$,
but as before, we can see that different choices of representatives give naturally isomorphic functors.
If $M$ and $N$ are two representation $RH$-modules,
then both $M\otimes_RN$ and $M\oplus N$ are also representation $RH$-modules.
Also $R$ is a representation $RH$-module with trivial $H$-action.
Then for any representation $RH$-module $M$,
its dual module $M^\circ$ is also a representation $RH$-module.
If $X$ is a finite $H$-set,
then the permutation module $RX$ is a representation $RH$-module.
As an analog to Proposition \ref{sUprops}, we have the following:\\

\begin{proposition}\label{tUprops}
Let $U,U'$ be right-free $(G,H)$-bisets, let $V$ be a right-free $(K,G)$-biset,
let $M$ and $N$ be representation $RH$-modules, and let $X$ be a finite $H$-set.
Then the following hold:
\begin{enumerate}
\item $t_U(R)\cong R$, where $R$ is the trivial module.
\item $t_U(M)\cong t_{U'}(M)$ as $RG$-modules whenever $U\cong U'$,
where the isomorphism is natural in $M$.
\item $t_U(M\otimes_RN)\cong t_U(M)\otimes_Rt_U(N)$ as $RG$-modules,
where the isomorphism is natural in both $M$ and $N$.
\item $t_{U\sqcup U'}(M)\cong t_U(M)\otimes_Rt_{U'}(M)$ as $RG$-modules,
where the isomorphism is natural in $M$.
\item $t_{V\times_GU}(M)\cong(t_V\circ t_U)(M)$ as $RK$-modules,
where the isomorphism is natural in $M$.
\item $t_U(M\oplus N)\cong\displaystyle\bigoplus_{V\in G\backslash\{V\subseteq_HU\}}\mathrm{Ind}_{G_V}^G(t_V(M)\otimes_Rt_{U-V}(N))$ as $RG$-modules,
where the isomorphism is natural in both $M$ and $N$.
\item $t_U(RX)\cong Rs_U(X)$,
where the isomorphism is natural in $X$.
\item $t_U(M^\circ)\cong t_U(M)^\circ$ as $RG$-modules,
where the isomorphism is natural in $M$.
\end{enumerate}
\end{proposition}
\begin{proof}
Proofs of the first seven properties are discussed in \cite{Bouc},
so here we only prove the last property.
Let $M$ be a representation $RH$-module,
and let $\{m_1,\dots,m_k\}$ be an $R$-basis of $M$.
Then $M^\circ$ has a dual $R$-basis $\{f_1,\dots,f_k\}$.
That is, $f_i\in\text{Hom}_R(M,R)$ such that $f_i(m_j)=\delta_{i,j}$
where $\delta_{i,j}=1_R$ if $i=j$ and $\delta_{i,j}=0_R$ if $i\neq j$.
Now also $\{m_{i_1}\otimes\cdots\otimes m_{i_n}:(i_1,\dots,i_n)\in\{1,\dots,k\}^n\}$
is an $R$-basis of $t_U(M)$.
For $(i_1,\dots,i_n)\in\{1,\dots,k\}^n$, let $f_{(i_1,\dots,i_n)}\in t_U(M)^\circ$
denote the dual element to $m_{i_1}\otimes\cdots\otimes m_{i_n}$.
And since $\{f_1,\dots,f_k\}$ is an $R$-basis of $M^\circ$,
then $\{f_{i_1}\otimes\cdots\otimes f_{i_n}:(i_1,\dots,i_n)\in\{1,\dots,k\}^n\}$
is an $R$-basis of $t_U(M^\circ)$.
The map $t_U(M^\circ)\longrightarrow t_U(M)^\circ$ that sends
$f_{i_1}\otimes\cdots\otimes f_{i_n}$ to $f_{(i_1,\dots,i_n)}$
is then an isomorphism of $R$-modules.
It is easy to check that this map preserves the $G$-action on both sides,
hence we conclude this is an isomorphism of $RG$-modules.
\end{proof}

\section{Tensor Induction Functions}\label{TenFun}
Here we show how tensor induction is used to define functions
between the other various representation rings and their ghost rings
by first starting with the trivial source ring.
So suppose $U$ is a right-free $(G,H)$-biset for two finite groups $G$ and $H$.
Let us fix a $p$-modular system $(K,\mathcal O,F)$ that is large enough for both $G$ and $H$.
Every finitely generated $FH$-module is automatically free over $F$,
hence every finitely generated $FH$-module is a representation module.
So we can apply $t_U$ to any finitely generated $FH$-module.
If $M$ is a trivial source $FH$-module,
then $M\mid FX$ for some finite $H$-set $X$.
Thus by property 6 of Proposition \ref{tUprops}, we see that $t_U(M)\mid t_U(FX)$.
And then by property 7, $t_U(FX)\cong Fs_U(X)$.
Hence $t_U(M)\mid Fs_U(X)$, proving that $t_U(M)$ is then a trivial source $FG$-module.
This shows that we have a function $T_F(U):T^+_F(H)\longrightarrow T_F(G), [M]\mapsto[t_U(M)],$
where $T_F^+(H)\subseteq T_F(H)$ denotes the semiring consisting of the classes
of trivial source $FH$-modules.
Since $[M]=[N]$ in $T_F(H)$ iff $M\cong N$ as $FH$-modules, this function is well-defined.
Properties 1 and 3 of Proposition \ref{tUprops} shows this function is multiplicative.
And as with $B(U)$, property 6 can be used to show that $T_F(U)$ is algebraic of degree $|U/H|$.
Then applying the theorem of Dress,
we get a function $T_F(H)\longrightarrow T_F(G)$, which we also denote by $T_F(U)$,
that is multiplicative and algebraic of degree $|U/H|$
such that $T_F(U)([M])=[t_U(M)]$ for every trivial source module $M$.
Property 7 implies that $T_F(U)\circ l_H=l_G\circ B(U)$.
Finally, property 8 of Proposition \ref{tUprops} can also be combined with the Theorem of Dress
to show that $T_F(U)$ preserves duals,
that is, $T_F(U)(a^\circ)=T_F(U)(a)^\circ$ for all $a\in T_F(H)$.

In the same way, we have a multiplicative function
$T_{\mathcal O}(U):T_{\mathcal O}(H)\longrightarrow T_{\mathcal O}(G)$,
which is algebraic of degree $|U/H|$ and also preserves duals.
Here, we also see that $T_{\mathcal O}(U)\circ l_H=l_G\circ B(U)$.
It is also clear that the isomorphisms
$T_{\mathcal O}(H)\cong T_F(H)$ and $T_{\mathcal O}(G)\cong T_F(G)$
commute with the maps $T_{\mathcal O}(U)$ and $T_F(U)$.

Next we describe in detail an extension of $T_{\mathcal O}(U)$ to the appropriate ghost rings.
So we want to define a multiplicative and algebraic function
$\tilde T_{\mathcal O}(U):\tilde T_{\mathcal O}(H)\longrightarrow\tilde T_{\mathcal O}(G)$
that completes the diagrams below to commutative ones:

\begin{center}
\begin{tabular}{cc}
\begin{tikzpicture}
\node (UL) {$T_{\mathcal O}(H)$};
\node (UR) [node distance=3.5cm, right of=UL] {$T_{\mathcal O}(G)$};
\node (LL) [below of=UL] {$\tilde{T}_{\mathcal O}(H)$};
\node (LR) [below of=UR] {$\tilde{T}_{\mathcal O}(G)$};
\draw [->] (UL) to node {$T_{\mathcal O}(U)$} (UR);
\draw [->] (LL) to node [swap] {$\tilde T_{\mathcal O}(U)$} (LR);
\draw [->] (UL) to node [swap] {$\tau_H$} (LL);
\draw [->] (UR) to node {$\tau_G$} (LR);
\end{tikzpicture}
&
\begin{tikzpicture}
\node (UL) {$\tilde B(H)$};
\node (UR) [node distance=3.5cm, right of=UL] {$\tilde B(G)$};
\node (LL) [below of=UL] {$\tilde{T}_{\mathcal O}(H)$};
\node (LR) [below of=UR] {$\tilde{T}_{\mathcal O}(G)$};
\draw [->] (UL) to node {$\tilde B(U)$} (UR);
\draw [->] (LL) to node [swap] {$\tilde T_{\mathcal O}(U)$} (LR);
\draw [->] (UL) to node [swap] {$\tilde l_H$} (LL);
\draw [->] (UR) to node {$\tilde l_G$} (LR);
\end{tikzpicture}
\end{tabular}
\end{center}

First we will show how, given a pair $(E,c)\in\mathscr{T}_p(G)$ and $u\in U$,
to define a corresponding element of $\mathscr{T}_p(H)$.
Since $E$ is $p$-hypo-elementary,
we see that $E\cap\prescript{u}{}{H}$ must also be $p$-hypo-elementary
with $O_p(E\cap\prescript{u}{}{H})=O_p(E)\cap\prescript{u}{}{H}$.
Recall that we have the map $\varphi_u$ which maps $E\cap\prescript{u}{}{H}$ onto $E^u$.
Also $\varphi_u$ maps $O_p(E)\cap\prescript{u}{}{H}$ onto $O_p(E)^u=O_p(E^u)$.
So we can define the surjective group morphism

$$\bar\varphi_u:E\cap\prescript{u}{}{H}/O_p(E)\cap\prescript{u}{}{H}\longrightarrow E^u/O_p(E^u),
\quad xO_p(E)\cap\prescript{u}{}{H}\mapsto\varphi_u(x)O_p(E^u).$$

Now let us set $e_u:=[E:O_p(E)(E\cap\prescript{u}{}{H})]$.
This then gives $\langle c^{e_u}\rangle=O_p(E)(E\prescript{u}{}{H})/O_p(E)\leq E/O_p(E)$.
We then have the canonical isomorphism:

$$\alpha_u:O_p(E)(E\cap\prescript{u}{}{H})/O_p(E)
\longrightarrow E\cap\prescript{u}{}{H}/O_p(E)\cap\prescript{u}{}{H},\quad
axO_p(E)\mapsto xO_p(E)\cap\prescript{u}{}{H},$$

\noindent where $a\in O_p(E), x\in E\cap\prescript{u}{}{H}$.
Thus we see that $E\cap\prescript{u}{}{H}/O_p(E)\cap\prescript{u}{}{H}$ is also cyclic,
generated by $\alpha_u(c^{e_u})$.
And since $\bar\varphi_u$ is surjective,
we see also that $E^u/O_p(E^u)$ is cyclic, 
generated by $(\bar\varphi_u\circ\alpha_u)(c^{e_u})$.
This shows that $E^u$ is also $p$-hypo-elementary.
Lastly, we set $f_u:=[O_p(E):O_p(E)\cap\prescript{u}{}{H}]$.
Then since $f_u$ is a power of $p$,
and $E^u/O_p(E^u)$ is a $p'$-group,
if we set $c^u:=(\bar\varphi_u\circ\alpha_u)(c^{e_u})^{f_u}$,
we also have $E^u/O_p(E^u)=\langle c^u\rangle$.
Hence $(E^u,c^u)\in\mathscr{T}_p(H)$.
More concretely, if $c=sO_p(E)$ for some $s\in E$,
then $s^{e_u}\in O_p(E)(E\cap\prescript{u}{}{H})$.
So suppose $s^{e_u}=ax$ with $a\in O_p(E), x\in E\cap\prescript{u}{}{H}$.
Then if $\varphi_u(x)=h$ (that is $xu=uh$),
we have $c^u=h^{f_u}O_p(E^u)$.
We can then define the following function:

$$\tilde T_{\mathcal O}(U):\tilde T_{\mathcal O}(H)\longrightarrow\tilde T_{\mathcal O}(G),\quad
\left(z_{(D,d)}\right)_{(D,d)\in\mathscr{T}_p(H)}\mapsto
\left(\prod_{u\in E\backslash U/H}z_{(E^u,c^u)}\right)_{(E,c)\in\mathscr{T}_p(G)}$$

\noindent We can then state the main theorem,
which includes a uniqueness statement on $\tilde T_{\mathcal O}(U)$.

\begin{theorem}\label{tUext}
Given a right-free $(G,H)$-biset $U$, the function 
$\tilde T_{\mathcal O}(U):\tilde T_{\mathcal O}(H)\longrightarrow\tilde T_{\mathcal O}(G)$
is a multiplicative, dual-preserving algebraic function of degree $|U/H|$
such that $\tilde T_{\mathcal O}(U)\circ\tau_H=\tau_G\circ T_{\mathcal O}(U)$
and $\tilde T_{\mathcal O}(U)\circ\tilde l_H=\tilde l_G\circ\tilde B(U)$.
Moreover, if $f:\tilde T_{\mathcal O}(H)\longrightarrow\tilde T_{\mathcal O}(G)$
is a multiplicative function satisfying $f\circ\tau_H=\tau_G\circ T_{\mathcal O}(U)$,
then $f=\tilde T_{\mathcal O}(U)$.
\end{theorem}
\begin{proof}
We will break up the proof into the following parts:
In part (a) we prove that $\tilde T_{\mathcal O}(U)$ is well-defined.
In part (b) we prove that the image of $\tilde T_{\mathcal O}(U)$
lies in $\tilde T_{\mathcal O}(G)$ as claimed.
It is then clear that $\tilde T_{\mathcal O}(U)$ is multiplicative, dual-preserving,
and algebraic of degree $|U/H|$.
In part (c) of the proof, we show that
$\tilde T_{\mathcal O}(U)\circ\tilde l_H=\tilde l_G\circ\tilde B(U)$.
In part (d) we prove that $\tilde T_{\mathcal O}(U)\circ\tau_H=\tau_G\circ T_{\mathcal O}(U)$.
And in part (e), we prove the uniqueness statement.

(a) We first show that $\tilde T_{\mathcal O}(U)$ is well-defined,
that is, doesn't depend on the choice of $E\backslash U/H$.
So let us fix an $(E,c)\in\mathscr{T}_p(G)$,
and suppose that $u,v\in U$ with $v=auh$ for some $a\in E$ and $h\in H$.
We claim that $E^v=(E^u)^h$ and $c^v=(c^u)^h$.
Thus $(E^u,c^u)$ and $(E^v,c^v)$ are conjugate in $\mathscr{T}_p(H)$,
and therefore $z_{(E^u,c^u)}=z_{(E^v,c^v)}$.
If $g\in E\cap\prescript{v}{}{H},$ we have

$$gv=v\varphi_v(g) \Rightarrow
g(auh)=(auh)\varphi_v(g) \Rightarrow
(a^{-1}ga)u=u(h\varphi_v(g)h^{-1})$$

\noindent From this we see that $E\cap\prescript{v}{}{H}=\prescript{a}{}{(E\cap\prescript{u}{}{H})}$
and $\varphi_v(g)=h^{-1}\varphi_u(a^{-1}ga)h.$
And since $O_p(E)$ is normal in $E$,
we also have $O_p(E)\cap\prescript{v}{}{H}=\prescript{a}{}{(O_p(E)\cap\prescript{u}{}{H})}$.
Hence $E^v=\varphi_v(E\cap\prescript{v}{}{H})
=\varphi_v(\prescript{a}{}{(E\cap\prescript{u}{}{H})})
=\varphi_u(E\cap\prescript{u}{}{H})^h=(E^u)^h$,
and similarly, $O_p(E^v)=O_p(E)^v=(O_p(E)^u)^h=O_p(E^u)^h$.
So we have
$$\bar\varphi_v(gO_p(E)\cap\prescript{v}{}{H})
=(\bar\varphi_u(a^{-1}gaO_p(E)\cap\prescript{u}{}{H}))^h$$
for $g\in E\cap\prescript{v}{}{H}.$
Now as we saw, $E\cap\prescript{v}{}{H}=\prescript{a}{}{(E\cap\prescript{u}{}{H})}.$
We therefore have
$O_p(E)(E\cap\prescript{v}{}{H})
=\prescript{a}{}{(O_p(E)(E\cap\prescript{u}{}{H}))}$.
Now both $O_p(E)(E\cap\prescript{v}{}{H})$ and $O_p(E)(E\cap\prescript{u}{}{H})$
are subgroups of $E$ containing $O_p(E)$.
So we get conjugate subgroups in $E/O_p(E)$.
But $E/O_p(E)$ is cyclic, so conjugation by $a$ is trivial.
Hence $O_p(E)(E\cap\prescript{v}{}{H})/O_p(E)=O_p(E)(E\cap\prescript{u}{}{H})/O_p(E)$.
Therefore, by the Correspondence Theorem, we see that 
$O_p(E)(E\cap\prescript{v}{}{H})=O_p(E)(E\cap\prescript{u}{}{H})$.
Then clearly $e_v=e_u$,
and we have $(\bar\varphi_v\circ\alpha_v)(c^{e_v})=h^{-1}(\bar\varphi_u\circ\alpha_u)(c^{e_u})h$.
Similarly, we can see $f_v=f_u$, which then implies $c^v=(c^u)^h$.
From this, we see that the product
$\prod_{u\in E\backslash U/H}z_{(E^u,c^u)}$ does not depend on the choice of $E\backslash U/H$.
Therefore $\tilde T_{\mathcal O}(U)$ is indeed a well-defined function.

(b) We next show that the image of $\tilde T_{\mathcal O}(U)$
is fixed by the actions of $G$ and $\Delta$.
To see that the image of $\tilde T_{\mathcal O}(U)$ is fixed by $G$,
we just need to notice that for $g\in G$ and $(E,c)\in\mathcal Q_{G,p}$,
we have $((\prescript{g}{}{E})^u,(\prescript{g}{}{c})^u)=(E^{g^{-1}u},c^{g^{-1}u})$ for any $u\in U$,
and if $(u_1,\dots,u_k)$ is an ordered set of representatives of $\prescript{g}{}{E}\backslash U/H$,
then $(g^{-1}u_1,\dots,g^{-1}u_k)$ is an ordered set of representatives of $E\backslash U/H$.

To see that the image of $\tilde T_{\mathcal O}(U)$ is fixed by the action of $\Delta$,
we pick a $\delta_i\in\Delta.$
Then for any $u\in U$, we have

$$(c^{i^*})^u
=(\bar\varphi_u\circ\alpha_u)((c^{i^*})^{e_uf_u})
=(\bar\varphi_u\circ\alpha_u)((c^{e_uf_u})^{i^*})
=(\bar\varphi_u\circ\alpha_u)(c^{e_uf_u})^{i^*}
=(c^u)^{i^*}.$$

\noindent Therefore

$$\delta_i\cdot\left(\prod_{u\in E\backslash U/H}z_{(E^u,c^u)}\right)_{(E,c)\in\mathscr{T}_p(G)}
=\left(\delta_i\left(\prod_{u\in E\backslash U/H}z_{(E^u,(c^{i^*})^u)}\right)\right)_{(E,c)\in\mathscr{T}_p(G)}$$
$$=\left(\prod_{u\in E\backslash U/H}\delta_i\left(z_{(E^u,(c^u)^{i^*})}\right)\right)_{(E,c)\in\mathscr{T}_p(G)}
=\left(\prod_{u\in E\backslash U/H}z_{(E^u,c^u)}\right)_{(E,c)\in\mathscr{T}_p(G)}$$

\noindent This last equality uses the fact that the domain of
$\tilde T_{\mathcal O}(U)$ is also fixed by $\Delta,$
giving $\delta_i(z_{(E^u,(c^u)^{i^*})})=z_{(E^u,c^u)}$.
So we have now shown that $\tilde T_{\mathcal O}(U)$ is a well-defined function
whose image lies in $\tilde T_{\mathcal O}(G)$ as stated.

(c) Now we show that
$\tilde T_{\mathcal O}(U)\circ\tilde l_H=\tilde l_G\circ\tilde B(U)$.
So suppose that $(n_S)_{S\in\mathscr{S}(H)}\in\tilde B(H)$.
We then have the following:

\begin{align*}
(\tilde l_G\circ\tilde B(U))((n_S)_{S\in\mathscr{S}(H)})
&=\tilde l_G\left(\left(\prod_{u\in T\backslash U/H}n_{T^u}\right)_{T\in\mathscr{S}(G)}\right)
=\left(\prod_{u\in E\backslash U/H}n_{E^u}\right)_{(E,c)\in\mathscr{T}_p(G)}\\
(\tilde T_{\mathcal O}(U)\circ\tilde l_H)((n_S)_{S\in\mathscr{S}(H)})
&=\tilde T_{\mathcal O}(U)((n_D)_{(D,d)\in\mathscr{T}_p(H)})
=\left(\prod_{u\in E\backslash U/H}n_{E^u}\right)_{(E,c)\in\mathscr{T}_p(G)}
\end{align*}

\noindent So we see that $\tilde T_{\mathcal O}(U)\circ\tilde l_H=\tilde l_G\circ\tilde B(U)$.

(d) Next we show that $\tilde T_{\mathcal O}(U)\circ\tau_H=\tau_G\circ T_{\mathcal O}(U)$.
We know $T_{\mathcal O}(H)$ is generated as an abelian group
by the classes of monomial $\mathcal OH$-modules (see \cite{Boltje} for instance).
Then since $\tilde T_{\mathcal O}(U)\circ\tau_H$ and $\tau_G\circ T_{\mathcal O}(U)$
are both algebraic of degree $|U/H|$,
we need only show they agree on the classes of monomial modules
and apply the Theorem of Dress to get equality.
A monomial $\mathcal OH$-module $M$ admits a decomposition
$M=\bigoplus_{x\in X}M_x$ into $\mathcal O$-submodules $M_x$,
which are free of rank one over $\mathcal O$
where $X$ is an $H$-set such that
$m\in M_x, h\in H,$ implies $hm\in M_{hx}$.
For $x\in X$, we let $H_x$ denote the corresponding stabilizer subgroup in $H$.
Since $M_x$ is $\mathcal O$-free of rank one,
$M_x=\mathcal O_{\psi_x}$ for some homomorphism
$\psi_x:H_x\longrightarrow\mathcal O^\times$.
Now suppose $Q\leq H$ is a $p$-subgroup.
Then $Q\cap H_x$ is contained in the kernel of each $\psi_x$
since 1 is the only $p$th root of unity in $\mathcal O^\times$.
We can see from \cite{BolKuls}, if we now consider $X^Q$ as an $N_H(Q)/Q$-set,
then $M(Q)=\bigoplus_{x\in X^Q}M_x$,
and so $M(Q)$ is a monomial $F[N_H(Q)/Q]$-module.
Now suppose that $t\in N_H(Q)$.
Then the Brauer character of $M(Q)$ at $tQ$ is $\sum_{x\in(X^Q)^{\langle t\rangle}}\psi_x(t)$.
Of course $(X^Q)^{\langle t\rangle}=X^{\langle Q,t\rangle}$,
and so we can see that

$$\tau_H([M])=\left(\sum_{x\in X^D}\psi_x(t)\right)_{(D,b)\in\mathscr{T}_p(H)}\quad
\text{where } b=tO_p(D).$$

Next we want to show that $t_U(M)$ is a monomial $\mathcal OG$-module,
where the underlying $G$-set is $X^n=s_U(X)$.
For any $(x_1,\dots,x_n)\in X^n$,
we can define the $\mathcal O$-module
$M_{(x_1,\dots,x_n)}=M_{x_1}\otimes_{\mathcal O}\cdots\otimes_{\mathcal O}M_{x_n}.$
Since each $M_{x_i}$ is a free $\mathcal O$-module of rank one, clearly

$$t_U(M)\cong\bigoplus_{(x_1,\dots,x_n)\in X^n}M_{(x_1,\dots,x_n)}$$

as $\mathcal O$-modules,
and $G$ permutes the summands according to its action on $X^n=s_U(X)$.
For $(x_1,\dots,x_n)\in X^n,$ we let $G_{(x_1,\dots,x_n)}$ denote its stabilizer subgroup of $G$.
Since $M_{(x_1,\dots,x_n)}$ is free of rank one over $\mathcal O$,
we know that $M_{(x_1,\dots,x_n)}=\mathcal O_{\psi_{(x_1,\dots,x_n)}}$
for some morphism $\psi_{(x_1,\dots,x_n)}:G_{(x_1,\dots,x_n)}\longrightarrow\mathcal O^\times$.
We next determine this morphism.
If $g\in G_{(x_1,\dots,x_n)}$ and $gu_i=u_{\pi(i)}h_i$ for $\pi\in S_n, h_i\in H,$ this means that

$$(x_1,\dots,x_n)=g\cdot(x_1,\dots,x_n)
=(h_{\pi^{-1}(1)}x_{\pi^{-1}(1)},\dots,h_{\pi^{-1}(n)}x_{\pi^{-1}(n)}).$$

\noindent Hence $x_i=h_{\pi^{-1}(i)}x_{\pi^{-1}(i)}$ for $i=1,\dots,n$.
In other words, $h_ix_i=x_{\pi(i)}$ for all $i$.
For each $i$, let us define $l_i\in\N$
to be the smallest natural number such that $\pi^{l_i}(i)=i$.
(Notice that $l_i=[\langle g\rangle:\langle g\rangle\cap\prescript{u_i}{}{H}]$.)
Under the restriction of the action of $S_n$ on $\{1,\dots,n\}$ to the subgroup $\langle\pi\rangle$,
we see that the orbit of $i$ is precisely $\{i,\pi(i),\pi^2(i),\dots,\pi^{l_i-1}(i)\}$.
So clearly $l_i$ is constant on the $\langle\pi\rangle$-orbits of $\{1,\dots,n\}$.
Now since $h_ix_i=x_{\pi(i)},$ we see that

$$h_{\pi^{l_i-1}(i)}\cdots h_{\pi(i)}h_ix_i
=h_{\pi^{l_i-1}(i)}\cdots h_{\pi(i)}x_{\pi(i)}
=\cdots
=h_{\pi^{l_i-1}(i)}x_{\pi^{l_i-1}(i)}
=x_{\pi^{l_i}(i)}
=x_i.$$

\noindent Hence $h_{\pi^{l_i-1}(i)}\cdots h_{\pi(i)}h_i\in H_{x_i}$,
and we can consider $\psi_{x_i}(h_{\pi^{l_i-1}(i)}\cdots h_{\pi(i)}h_i)$,
which is independent of the $\langle\pi\rangle$-orbit of $i$.
Now to determine $\psi_{(x_1,\dots,x_n)}(g),$
we pick a basis of $M_{(x_1,\dots,x_n)}$ in the following way:
Fix an $i\in\{1,\dots,n\}$ and choose any $\mathcal O$-generator $e_{x_i}$ of $M_{x_i}$.
Then for $j=1,\dots,l_i-1,$
we recursively define $e_{x_{\pi^j(i)}}:=h_{\pi^{j-1}(i)}e_{x_{\pi^{j-1}(i)}}\in
M_{h_{\pi^{j-1}(i)}x_{\pi^{j-1}(i)}}=M_{x_{\pi^j(i)}}.$
So we've chosen an $\mathcal O$-generator of $M_{x_{i'}}$
for all $i'$ in the $\langle\pi\rangle$-orbit of $i$.
We then continue this process for the other $\langle\pi\rangle$-orbits of $\{1,\dots,n\}$.
Then $e_{x_1}\otimes\cdots\otimes e_{x_n}$ is an
$\mathcal O$-basis element of $M_{(x_1,\dots,x_n)}$.
So $g\cdot(e_{x_1}\otimes\cdots\otimes e_{x_n})
=\psi_{(x_1,\dots,x_n)}(g)(e_{x_1}\otimes\cdots\otimes e_{x_n}).$
But by construction, we see that

$$g\cdot(e_{x_1}\otimes\cdots\otimes e_{x_n})
=h_{\pi^{-1}(1)}e_{x_{\pi^{-1}(1)}}\otimes\cdots\otimes h_{\pi^{-1}(n)}e_{x_{\pi^{-1}(n)}}
=\prod_{i\in\langle\pi\rangle\backslash\{1,\dots,n\}}\psi_{x_i}(h_{\pi^{l_i-1}(i)}\cdots h_{\pi(i)}h_i)
(e_{x_1}\otimes\cdots\otimes e_{x_n})$$

\noindent where $\langle\pi\rangle\backslash\{1,\dots,n\}$ denotes a full set of representatives
of the $\langle\pi\rangle$-orbits of $\{1,\dots,n\}$.
This then shows that if $g\in G_{(x_1,\dots,x_n)}$ with $gu_i=u_{\pi(i)}h_i,$ then we have

$$\psi_{(x_1,\dots,x_n)}(g)
=\prod_{i\in\langle\pi\rangle\backslash\{1,\dots,n\}}\psi_{x_i}(h_{\pi^{l_i-1}(i)}\cdots h_{\pi(i)}h_i).$$

We are now ready to compute the image of $[M]$ under both relevant compositions:

\begin{align*}
(\tau_G\circ T_{\mathcal O}(U))([M])
&=\tau_G\left(\left[\bigoplus_{(x_1,\dots,x_n)\in X^n}M_{(x_1,\dots,x_n)}\right]\right)
=\left(\sum_{(x_1,\dots,x_n)\in(X^n)^E}\psi_{(x_1,\dots,x_n)}(s)\right)_{(E,c)\in\mathscr{T}_p(G)}\\
&=\left(\sum_{(x_1,\dots,x_n)\in(X^n)^E}\prod_{i\in\langle\pi\rangle\backslash\{1,\dots,n\}}\psi_{x_i}(s_{\pi^{l_i-1}(i)}\cdots s_{\pi(i)}s_i)\right)_{(E,c)\in\mathscr{T}_p(G)}\\
(\tilde T_{\mathcal O}(U)\circ\tau_H)([M])
&=\tilde T_{\mathcal O}(U)\left(\left(\sum_{x\in X^D}\psi_x(t)\right)_{(D,tO_p(D))\in\mathscr{T}_p(H)}\right)
=\left(\prod_{u\in E\backslash U/H}\sum_{x\in X^{E^u}}\psi_x(h_u^{f_u})\right)_{(E,c)\in\mathscr{T}_p(G)}
\end{align*}

\noindent where for $(E,c)\in\mathscr{T}_p(G)$,
we write $c=sO_p(E)$ with $su_i=u_{\pi(i)}s_i$,
and for $u\in U$, we write $c^u=h_u^{f_u}O_p(E^u)$.
So to prove that these two compositions agree,
let us fix one such $(E,c)\in\mathscr{T}_p(G)$.
We can then partition $\{1,\dots,n\}$ by some $\lambda=\{\lambda_1,\dots,\lambda_k\}$
such that $u_i$ and $u_{i'}$ are in the same $E\backslash U/H$-class iff $i,i'\in\lambda_j$
for some $j\in\{1,\dots,k\}$ where $k=|E\backslash U/H|$.
We then define a function $\nu:\{1,\dots,n\}\longrightarrow\{1,\dots,k\}$
by $\nu(i)=j$ whenever $i\in\lambda_j$.
And let $\eta:\{1,\dots,k\}\longrightarrow\{1,\dots,n\}$ be the section of $\nu$
defined by sending $j$ to the smallest $i$ such that $i\in\lambda_j$.
We then set $v_j:=u_{\nu(j)}$
so that $\{v_1,\dots,v_k\}$ is a complete set of representatives of $E\backslash U/H$.
Now if $i\in\lambda_j$, we have $u_i=b_iv_jk_i$ for some $b_i\in E$ and $k_i\in H$.
We can then define the following functions, which are mutual inverses of each other:

\begin{align*}
\beta&:\prod_{j=1}^kX^{E^{v_j}}\longrightarrow(X^n)^E,\quad
(y_1,\dots,y_k)\mapsto(x_1,\dots,x_n)\text{ where }x_i=k_i^{-1}y_{\nu(i)},\\
\alpha&:(X^n)^E\longrightarrow\prod_{j=1}^kX^{E^{v_j}},\quad
(x_1,\dots,x_n)\mapsto(x_{\eta(1)},\dots,x_{\eta(k)}).
\end{align*}

Now $E/O_p(E)=\langle c\rangle$ is a $p'$-group,
and we can assume $c=sO_p(E)$ for some $p'$-element $s\in E$.
Also we assume $su_i=u_{\pi(i)}s_i$ for some $\pi\in S_n,$ and $s_i\in H$.
For $j=1,\dots,k,$ we set $e_j:=e_{v_j}=[E:O_p(E)(E\cap\prescript{v_j}{}{H})],$
and $f_j:=f_{v_j}=[O_p(E)(E\cap\prescript{v_j}{}{H}):E\cap\prescript{v_j}{}{H}]$.
Then $s^{e_j}=a_jx_j$ for some $a_j\in O_p(E),x_j\in E\cap\prescript{v_j}{}{H}$.
We then set $h_j:=\varphi_{v_j}(x_j)\in E^{v_j}$, and we can define the function

$$\delta_j:X^{E^{v_j}}\longrightarrow\mathcal O^\times,
\quad x\mapsto\psi_x(h_j)^{f_j}.$$

\noindent Since $s\in E$, notice we always have $\nu(i)=\nu(\pi(i))$.
In other words, if $i\in\lambda_j$, then also $\pi(i)\in\lambda_j$.
Hence the action of $\langle\pi\rangle$ on $\{1,\dots,n\}$
induces an action on each $\lambda_j$.
Then if we let $\langle\pi\rangle\backslash\lambda_j$
denote a complete set of representatives of the $\langle\pi\rangle$-orbits
of $\lambda_j$ under this restriction, we can define the function

$$\epsilon_j:(X^n)^E\longrightarrow\mathcal O^\times,\quad
(x_1,\dots,x_n)\mapsto\prod_{i\in\langle\pi\rangle\backslash\lambda_j}\psi_{x_i}(s_{\pi^{l_i-1}(i)}\cdots s_{\pi(i)}s_i).$$

Now for each $j\in\{1,\dots,k\}$
we let $p_j$ denote the projection of $\prod_{j=1}^kX^{E^{v_j}}$ onto $X^{E^{v_j}}$.
We claim that $\epsilon_j\circ\beta=\delta_j\circ p_j$ for each $j$.
To see this, suppose that $(y_1,\dots,y_k)\in\prod_{j=1}^kX^{E^{v_j}}$.
Then

$$(\epsilon_j\circ\beta)(y_1,\dots,y_k)=\epsilon_j(x_1,\dots,x_n)
=\prod_{i\in\langle\pi\rangle\backslash\lambda_j}\psi_{x_i}(s_{\pi^{l_i-1}(i)}\cdots s_{\pi(i)}s_i),$$

\noindent where $x_i=k_i^{-1}y_j$ for $i\in\lambda_j$.
This means that $\psi_{x_i}=\psi_{k_i^{-1}y_j}=\prescript{k_i^{-1}}{}{\psi_{y_j}}$.
Hence

$$(\epsilon_j\circ\beta)(y_1,\dots,y_k)
=\prod_{i\in\langle\pi\rangle\backslash\lambda_j}\prescript{k_i^{-1}}{}{\psi_{y_j}}(s_{\pi^{l_i-1}(i)}\cdots s_{\pi(i)}s_i)
=\prod_{i\in\langle\pi\rangle\backslash\lambda_j}\psi_{y_j}(k_is_{\pi^{l_i-1}(i)}\cdots s_{\pi(i)}s_ik_i^{-1})$$

\noindent On the other hand, $(\delta_j\circ p_j)(y_1,\dots,y_k)=\psi_{y_j}(h_j)^{f_j}$.
Now for $i\in\lambda_j$, we see that

$$s^{l_i}u_i=s^{l_i-1}u_{\pi(i)}s_i=s^{l_i-2}u_{\pi^2(i)}s_{\pi(i)}s_i
=\cdots=u_{\pi^{l_i}(i)}s_{\pi^{l_i-1}(i)}\cdots s_{\pi(i)}s_i
=u_is_{\pi^{l_i-1}(i)}\cdots s_{\pi(i)}s_i.$$

\noindent Hence $s^{l_i}\in E\cap\prescript{u_i}{}{H}$
with $\varphi_{u_i}(s^{l_i})=s_{\pi^{l_i-1}(i)}\cdots s_{\pi(i)}s_i$.
Since $u_i=b_iv_jk_i,$
then $b_i^{-1}s^{l_i}b_i\in E\cap\prescript{v_j}{}{H}$
with $\varphi_{v_j}(b_i^{-1}s^{l_i}b_i)=k_i\varphi_{u_i}(s^{l_i})k_i^{-1}$.
On the other hand, we notice
$s^{l_i}\in E\cap\prescript{u_i}{}{H}\leq O_p(E)(E\cap\prescript{u_i}{}{H})
=O_p(E)(E\cap\prescript{v_j}{}{H})$.
But $e_j$ is the smallest natural number such that
$s^{e_j}\in O_p(E)(E\cap\prescript{v_j}{}{H})$.
Hence $s^{l_i}=(s^{e_j})^{m_i}$ for some $m_i\in\N$.
So if we set $m:=[E:O_p(E)]$,
then the assumption that $s$ is a $p'$ element implies the order of $s$ is $m$.
Thus $l_i\equiv e_jm_i\pmod m$.
Now since $a_j\in O_p(E)\trianglelefteq E,$ we see that 
$s^{l_i}=(s^{e_j})^{m_i}=(a_jx_j)^{m_i}=a_j'x_j^{m_i},$
for some $a_j'\in O_p(E)$.
We then compute

\begin{align*}
(\bar\varphi_{v_j}\circ\alpha_{v_j})(s^{l_i}O_p(E))
&=\bar\varphi_{v_j}(\alpha_{v_j}(a_j'x_j^{m_i}O_p(E)))
=\bar\varphi_{v_j}(x_j^{m_i}O_p(E)\cap\prescript{v_j}{}{H})\\
&=\varphi_{v_j}(x_j^{m_i})O_p(E^{v_j})
=\varphi_{v_j}(x_j)^{m_i}O_p(E^{v_j})
=h_j^{m_i}O_p(E^{v_j})
\end{align*}

\noindent But on the other hand,

$$(\bar\varphi_{v_j}\circ\alpha_{v_j})(s^{l_i}O_p(E))
=\varphi_{v_j}(b_i^{-1}s^{l_i}b_i)O_p(E^{v_j})
=k_i\varphi_{u_i}(s^{l_i})k_i^{-1}O_p(E^{v_j})
=k_is_{\pi^{l_i-1}(i)}\cdots s_{\pi(i)}s_ik_i^{-1}O_p(E^{v_j}).$$

\noindent Now $O_p(E^{v_j})$ is contained in the kernel of $\psi_{y_j}$,
so $\psi_{y_j}(k_is_{\pi^{l_i-1}(i)}\cdots s_{\pi(i)}s_ik_i^{-1})=\psi_{y_j}(h_j)^{m_i}.$
Therefore

$$(\epsilon_j\circ\beta)(y_1,\dots,y_k)
=\prod_{i\in\langle\pi\rangle\backslash\lambda_j}\psi_{y_j}(h_j)^{m_i}
=\psi_{y_j}(h_j)^{\sum_{i\in\langle\pi\rangle\backslash\lambda_j}m_i}.$$

\noindent For each $i\in\lambda_j$, we have $l_i\equiv e_jm_i\pmod m,$ and
$m=[E:O_p(E)]=[E:O_p(E)(E\cap\prescript{v_j}{}{H})][O_p(E)(E\cap\prescript{v_j}{}{H}):O_p(E)]$.
So if we set $d_j:=[O_p(E)(E\cap\prescript{v_j}{}{H}):O_p(E)],$
then $m=e_jd_j$ or $\frac{m}{e_j}=d_j$.
Now we see that $e_j\sum_{i\in\langle\pi\rangle\backslash\lambda_j}m_i
=\sum_{i\in\langle\pi\rangle\backslash\lambda_j}e_jm_i,$
and since $e_jm_i\equiv l_i\pmod m,$
we have $e_j\sum_{i\in\langle\pi\rangle\backslash\lambda_j}m_i\equiv
\sum_{i\in\langle\pi\rangle\backslash\lambda_j}l_i\pmod m$.
But clearly $\sum_{i\in\langle\pi\rangle\backslash\lambda_j}l_i=|\lambda_j|$.
We can see that $E/E\cap\prescript{v_j}{}{H}=\{b_iE\cap\prescript{v_j}{}{H}:i\in\lambda_j\}$,
and therefore $|\lambda_j|=[E:E\cap\prescript{v_j}{}{H}]=
[E:O_p(E)(E\cap\prescript{v_j}{}{H})][O_p(E)(E\cap\prescript{v_j}{}{H}):E\cap\prescript{v_j}{}{H}]
=e_jf_j.$
So putting all of this together,
we see that $e_j\sum_{i\in\langle\pi\rangle\backslash\lambda_j}m_i\equiv e_jf_j\pmod m$.
And since $\frac{m}{e_j}=d_j$,
we get $\sum_{i\in\langle\pi\rangle\backslash\lambda_j}m_i\equiv f_j\pmod {d_j}$.
Now since $h_j=\varphi_{v_j}(x_j),$
we see that $o(h_j)\mid o(x_j)=o(s^{e_j})=\frac{m}{e_j}=d_j$.
Hence $\psi_{y_j}(h_j)$ is a $d_j$th root of unity in $\mathcal O^\times$.
Then since $\sum_{i\in\langle\pi\rangle\backslash\lambda_j}m_i\equiv f_j\pmod{d_j},$
we see that

$$(\epsilon_j\circ\beta)(y_1,\dots,y_k)
=\psi_{y_j}(h_j)^{\sum_{i\in\langle\pi\rangle\backslash\lambda_j}m_i}
=\psi_{y_j}(h_j)^{f_j}
=(\delta_j\circ p_j)(y_1,\dots,y_k).$$

So we have that $\epsilon_j\circ\beta=\delta_j\circ p_j$ for $j=1,\dots,k$.
So finally we see the following:

\begin{align*}
&\sum_{(x_1,\dots,x_n)\in(X^n)^E}\prod_{i\in\langle\pi\rangle\backslash\{1,\dots,n\}}\psi_{x_i}(s_{\pi^{l_i-1}(i)}\cdots s_{\pi(i)}s_i)
=\sum_{(x_1,\dots,x_n)\in(X^n)^E}\prod_{j=1}^k\epsilon_j(x_1,\dots,x_n)\\
&=\sum_{(x_1,\dots,x_n)\in(X^n)^E}\prod_{j=1}^k(\epsilon_j\circ\beta\circ\alpha)(x_1,\dots,x_n)
=\sum_{(x_1,\dots,x_n)\in(X^n)^E}\prod_{j=1}^k(\delta_j\circ p_j)(x_{\eta(1)},\dots,x_{\eta(k)})\\
&=\sum_{(x_1,\dots,x_n)\in(X^n)^E}\prod_{j=1}^k\delta_j(x_{\eta(j)})
=\sum_{(y_1,\dots,y_k)\in\prod_{j=1}^kX^{E^{v_j}}}\prod_{j=1}^k\delta_j(y_j)
\end{align*}

\noindent Then by distributivity,
this is equal to $\prod_{j=1}^k\sum_{x\in X^{E^{v_j}}}\delta_j(x)
=\prod_{j=1}^k\sum_{x\in X^{E^{v_j}}}\psi_x(h_j)^{f_j}$.
This proves that the $(E,c)$-components of $(\tau_G\circ T_{\mathcal O}(U))([M])$
and $(\tilde T_{\mathcal O}(U)\circ\tau_h)([M])$ are equal.
So we see that $\tau_G\circ T_{\mathcal O}(U)$ and $\tilde T_{\mathcal O}(U)\circ\tau_H$
agree on the classes of monomial modules.
Then since both compositions are algebraic,
we can conclude that $\tilde T_{\mathcal O}(U)\circ\tau_H=\tau_G\circ T_{\mathcal O}(U).$

(e) Finally, we show that this extension $\tilde T_{\mathcal O}(U)$ is the unique
multiplicative extension of $T_{\mathcal O}(U)$.
So suppose that also $f:\tilde T_{\mathcal O}(H)\longrightarrow\tilde T_{\mathcal O}(G)$
is multiplicative and satisfies $f\circ\tau_h=\tau_G\circ T_{\mathcal O}(U)$.
We know that the ghost map $\tau_H$ is injective with finite cokernel annihilated by $|H|$.
So pick $b\in\tilde T_{\mathcal O}(H)$.
Then $|H|b=\tau_H(a)$ for some $a\in T_{\mathcal O}(H)$.
Hence
$$f(|H|b)=f(\tau_H(a))=(f\circ\tau_H)(a)=(\tau_G\circ T_{\mathcal O}(U))(a)=
(\tilde T_{\mathcal O}(U)\circ\tau_H)(a)=\tilde T_{\mathcal O}(U)(\tau_h(a))=
\tilde T_{\mathcal O}(U)(|H|b).$$
Now if $\mathbf 1_H$ denotes the identity of $\tilde T_{\mathcal O}(H)$,
then we can similarly see that $f(|H|\mathbf 1_H)=\tilde T_{\mathcal O}(U)(|H|\mathbf 1_H).$
So we have $$\tilde T_{\mathcal O}(U)(|H|\mathbf1_H)f(b)=
f(|H|\mathbf1_H)f(b)=f(|H|b)=\tilde T_{\mathcal O}(U)(|H|b)=
\tilde T_{\mathcal O}(U)(|H|\mathbf1_H)\tilde T_{\mathcal O}(U)(b).$$
Let us set $B:=\tilde T_{\mathcal O}(U)(|H|\mathbf1_H)\in\tilde T_{\mathcal O}(H)$.
An easy computation shows that $B=(|H|^{|E\backslash U/H|})_{(E,c)\in\mathscr{T}_p(G)}.$
Hence $B$ has a nonzero integer in each component.
Therefore $Bf(b)=B\tilde T_{\mathcal O}(U)(b)$ implies $f(b)=\tilde T_{\mathcal O}(U)(b)$.
Thus we conclude that $\tilde T_{\mathcal O}(U)$ is unique among multiplicative functions
extending $T_{\mathcal O}(U)$ to ghost rings.
\end{proof}

If we set $\tilde T_F(U)=\tilde T_{\mathcal O}(U)$,
then the previous theorem shows also that $\tilde T_F(U)$ extends $T_F(U)$ to ghost rings.
So we have defined tensor induction functions between Burnside rings and trivial source rings,
and extended them to ghost rings.
We next deal with the character ring.
We know that two $KH$-modules have the same image in $R_K(H)$
if and only if they are isomorphic as $KH$-modules.
So if we let $R_K(H)^+\subseteq R_K(H)$
denote the semiring generated by the classes of $KH$-modules,
then we see the functor $t_U$ induces a map $R_K(U):R_K(H)^+\longrightarrow R_K(G)$
by $R_K(U)([M])=[t_U(M)]$.
As with $T_{\mathcal O}(U)$, we see that Proposition \ref{tUprops}
can be used to show that $R_K(U)$ is multiplicative, preserves duals,
and is algebraic of degree $|U/H|$.
Hence this map extends uniquely to a multiplicative, dual-preserving, algebraic function
$R_K(H)\longrightarrow R_K(G)$, which we will also denote by $R_K(U)$.
Before continuing, we first prove the following,
which was previously proved in a special case and via different methods
by David Gluck and Marty Isaacs in \cite{GluIsa}:
\begin{proposition}\label{inducedchar}
If $M$ is a $KH$-module with character $\chi$,
then $t_U(M)$ has character $\chi^U$, where for $x\in G$,
$$\chi^U(x)=\prod_{u\in\langle x\rangle\backslash U/H}\chi(\varphi_u(x^{n_u})),$$
where $n_u=[\langle x\rangle:\langle x\rangle\cap\prescript{u}{}{H}]$.
\end{proposition}
\begin{proof}
To see first that this function is well-defined,
suppose that $u,v\in U$ with $v=x^muh$ for some $m\in\Z,h\in H$.
Then $\prescript{v}{}{H}=\prescript{x^m}{}{(\prescript{u}{}{H})}$.
So $n_v=[\langle x\rangle:\langle x\rangle\cap\prescript{x^m}{}{(\prescript{u}{}{H})}]
=[\langle x\rangle:\prescript{x^m}{}{(\langle x\rangle\cap\prescript{u}{}{H})}]
=[\langle x\rangle:\langle x\rangle\cap\prescript{u}{}{H}]=n_u.$
Then $$\varphi_v(x^{n_v})=\varphi_{x^muh}(x^{n_u})
=h^{-1}\varphi_u(x^mx^{n_u}x^{-m})h=h^{-1}\varphi_u(x^{n_u})h.$$
So $\varphi_v(x^{n_v})$ and $\varphi_u(x^{n_u})$ are conjugate in $H$,
and therefore $\chi(\varphi_v(x^{n_v}))=\chi(\varphi_u(x^{n_u}))$.
Hence $\chi^U(x)$ doesn't depend on the choice of $\langle x\rangle\backslash U/H$.
Now fix a set of representatives $U/H=(u_1,\dots,u_n)$
and suppose that $xu_i=u_{\pi(i)}h_i$ for some $\pi\in S_n$ and $h_i\in H$.
Now $\langle\pi\rangle\leq S_n$ acts on $\{1,\dots,n\}$,
and if $\langle\pi\rangle\backslash\{1,\dots,n\}=\{i_1,\dots,i_k\}$,
then $\langle x\rangle\backslash U/H=(u_{i_1},\dots,u_{i_k})$.
Let us fix a $K$-basis $m_1,\dots,m_r$ of $M$.
Suppose also that $h_im_j=\sum_{k_i=1}^r\alpha_{ijk_i}m_{k_i}$.
We can see for each $i$ that $\varphi_{u_i}(x^{n_{u_i}})=h_{\pi^{n_{u_i}-1}(i)}\cdots h_{\pi(i)}h_i$
and therefore $$\chi(\varphi_{u_i}(x^{n_{u_i}}))
=\sum_{(k_i,k_{\pi(i)},\dots,k_{\pi^{n_{u_i}-1}(i)})\in\{1,\dots,r\}^{n_{u_i}}}
\prod_{j=0}^{n_{u_i}-1}\alpha_{\pi^j(i)k_{\pi^j(i)}k_{\pi^{j+1}(i)}}.$$
So by distributivity, we have
\begin{align*}
\prod_{u\in\langle x\rangle\backslash U/H}\chi(\varphi_u(x^{n_u}))
&=\sum_{(k_1,\dots,k_n)\in\{1,\dots,r\}^n}\prod_{i\in\langle\pi\rangle\backslash\{1,\dots,n\}}
\prod_{j=0}^{n_{u_i}-1}\alpha_{\pi^j(i)k_{\pi^j(i)}k_{\pi^{j+1}(i)}}\\
&=\sum_{(k_1,\dots,k_n)\in\{1,\dots,r\}^n}\prod_{i=1}^n\alpha_{ik_ik_{\pi(i)}}\\
&=\sum_{(k_1,\dots,k_n)\in\{1,\dots,r\}^n}\prod_{i=1}^n\alpha_{\pi^{-1}(i)k_{\pi^{-1}(i)}k_i}
\end{align*}
On the other hand, $\{m_{k_1}\otimes\cdots\otimes m_{k_n}:(k_1,\dots,k_n)\in\{1,\dots,r\}^n\}$
is then a $K$-basis of $t_U(M)$.
Then for $(i_1,\dots,i_n)\in\{1,\dots,r\}^n$, we see that
\begin{align*}
x\cdot(m_{i_1}\otimes\cdots\otimes m_{i_n})
&=h_{\pi^{-1}(1)}m_{i_{\pi^{-1}(1)}}\otimes\cdots\otimes h_{\pi^{-1}(n)}m_{i_{\pi^{-1}(n)}}\\
&=(\sum_{k_1=1}^r\alpha_{\pi^{-1}(1)i_{\pi^{-1}(1)}k_1}m_{k_1})\otimes\cdots\otimes(\sum_{k_n=1}^r\alpha_{\pi^{-1}(n)i_{\pi^{-1}(n)}k_n}m_{k_n})\\
&=\sum_{(k_1,\dots,k_n)\in\{1,\dots,r\}^n}\prod_{j=1}^n\alpha_{\pi^{-1}(j)i_{\pi^{-1}(j)}k_j}m_{k_1}\otimes\cdots\otimes m_{k_n}
\end{align*}
So we see that the character of $t_U(M)$ at $x$ is precisely
$\sum_{(k_1,\dots,k_n)\in\{1,\dots,r\}^n}\prod_{i=1}^n\alpha_{\pi^{-1}(i)k_{\pi^{-1}(i)}k_i}=\chi^U(x)$.
\end{proof}
So the map $R_K(U)$ can be described as the function $\chi\mapsto\chi^U$ on characters.
If $M$ is a trivial source $\mathcal OH$-module,
then clearly $(R_K(U)\circ c_H)([M])=(c_G\circ T_{\mathcal O}(U))([M])$
since both are equal to the character of $t_U(M)$.
Then since both $R_K(U)\circ c_H$ and $c_G\circ T_{\mathcal O}(U)$
are algebraic (of degree $|U/H|$),
we can conclude by the Theorem of Dress
that $R_K(U)\circ c_H=c_G\circ T_{\mathcal O}(U)$.

It is now easy to see how to extend $R_K(U)$ to ghost rings.
We define the function
$$\tilde R_K(U):\tilde R_K(H)\longrightarrow\tilde R_K(G),\quad
(w_h)_{h\in\mathscr{E}(H)}\mapsto
\left(\prod_{u\in\langle x\rangle\backslash U/H}w_{\varphi_u(x^{n_u})}\right)_{x\in\mathscr{E}(G)}.$$
It's immediately clear then that $\tilde R_K(U)$ is multiplicative and algebraic of degree $|U/H|$
and also that $\tilde R_K(U)\circ\varepsilon_H=\varepsilon_G\circ R_K(U)$.
That is, $\tilde R_K(U)$ is a multiplicative extension of $R_K(U)$.
Since $\varepsilon_H$ is injective with finite cokernel,
we see also that $\tilde R_K(U)$ is the unique multiplicative extension of $R_K(U)$
in much the same way we proved
$\tilde T_{\mathcal O}(U)$ is the unique multiplicative extension of $T_{\mathcal O}(U)$.
Also if $(z_{(D,b)})_{(D,b)\in\mathscr T_p(H)}\in\tilde T_{\mathcal O}(H)$,
then $(\tilde R_K(U)\circ\tilde c_H)((z_{(D,b)})_{(D,b)\in\mathscr T_p(H)})
=(\prod_{u\in\langle x\rangle\backslash U/H}
z_{(\langle x\rangle^u,\varphi_u(x^{n_u})\langle x_p\rangle^u)})_{x\in\mathscr{E}(G)}
=(\tilde c_G\circ\tilde T_{\mathcal O}(U))((z_{(D,b)})_{(D,b)\in\mathscr T_p(H)}).$
Hence $\tilde R_K(U)\circ\tilde c_H=\tilde c_G\circ\tilde T_{\mathcal O}(U)$.

Lastly we want to describe the function $R_F(U):R_F(H)\longrightarrow R_F(G)$
induced by the tensor induction functor
and an extension $\tilde R_F(U):\tilde R_F(H)\longrightarrow\tilde R_F(G)$.
Essentially we want to complete the bottom face of the following commutative diagram:

\begin{equation}
\label{tocomplete}
\raisebox{-0.5\height}{\includegraphics{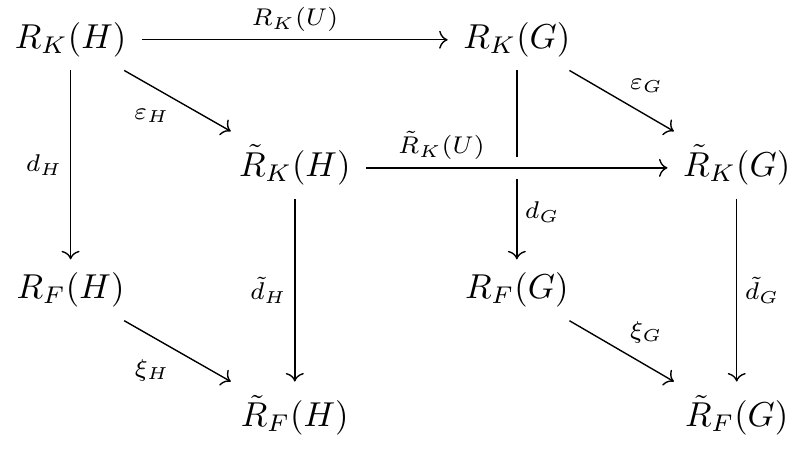}}
\end{equation}

Similar to previous constructions,
we first let $R^+_F(H)$ denote the subring of $R_F(H)$ consisting of the classes of $FH$-modules.
We want to define $R_F(U):R_F^+(H)\longrightarrow R_F(G)$ by $R_F(U)([M])=[t_U(M)]$,
but it is not immediately clear that this is well-defined.
So suppose that $M$ and $N$ are two $FH$-modules such that $[M]=[N]$ in $R^+_F(H)$.
This is equivalent to having
$\mathrm{Res}_T^H(M)\cong\mathrm{Res}_T^H(N)$ for all $p'$-subgroups $T\leq H$.
To see that $R_F(U)$ is well-defined as above, it suffices to show that
$\mathrm{Res}_S^G(t_U(M))\cong\mathrm{Res}_S^G(t_U(N))$ for all $p'$-subgroups $S\leq G$.
So let us fix one $p'$-subgroup $S\leq G$.
Notice that the restriction functor $\mathrm{Res}_S^G$ is just the functor $t_V$,
where $V$ is the elementary restriction $(S,G)$-biset.
Thus $\mathrm{Res}_S^G(t_U(M))=(t_V\circ t_U)(M)=t_{V\times_GU}(M)$.
Let us first write $U=U_1\sqcup\cdots\sqcup U_t$ as a disjoint union of transitive $(G,H)$-bisets,
and then pick a $u_i\in U_i$.
Now each $U_i$ is right-free, and we therefore have isomorphisms
$\bar\varphi_{u_i}:\prescript{u_i}{}{H}/\prescript{u_i}{}{\mathbf 1}\stackrel{\sim}{\longrightarrow}G^{u_i},
g\prescript{u_i}{}{\mathbf 1}\mapsto\varphi_{u_i}(g).$
We can write each transitive $U_i$ as a composition of elementary bisets in the following way:
$$U_i\cong\mathrm{Ind}^G_{\prescript{u_i}{}{H}}\circ\mathrm{Inf}^{\prescript{u_i}{}{H}}_{\prescript{u_i}{}{H}/\prescript{u_i}{}{\mathbf 1}}\circ\mathrm{Iso}(\bar\varphi_{u_i}^{-1})\circ\mathrm{Res}^H_{G^{u_i}}.$$
We then compose each $U_i$ with $V=\mathrm{Res}_S^G$.
First we see by the Mackey formula that
$$\mathrm{Res}_S^G\circ\mathrm{Ind}_{\prescript{u_i}{}{H}}^G\cong\coprod_{x\in S\backslash G/\prescript{u_i}{}{H}}\mathrm{Ind}^S_{S\cap\prescript{xu_i}{}{H}}\circ\mathrm{Iso}(c_x)\circ\mathrm{Res}^{\prescript{u_i}{}{H}}_{S^x\cap\prescript{u_i}{}{H}}$$
where $c_x:S^x\cap\prescript{u_i}{}{H}\stackrel{\sim}{\longrightarrow}S\cap\prescript{xu_i}{}{H}$
is the isomorphism induced by conjugation by $x$.
Then for any $x\in G$, we have
$$\mathrm{Res}^{\prescript{u_i}{}{H}}_{S^x\cap\prescript{u_i}{}{H}}\circ\mathrm{Inf}^{\prescript{u_i}{}{H}}_{\prescript{u_i}{}{H}/\prescript{u_i}{}{\mathbf 1}}\cong\mathrm{Inf}^{S^x\cap\prescript{u_i}{}{H}}_{S^x\cap\prescript{u_i}{}{H}/S^x\cap\prescript{u_i}{}{\mathbf 1}}\circ\mathrm{Iso}(\gamma_x^{-1})\circ\mathrm{Res}^{\prescript{u_i}{}{H}/\prescript{u_i}{}{\mathbf 1}}_{(S^x\cap\prescript{u_i}{}{H})\prescript{u_i}{}{H}/\prescript{u_i}{}{\mathbf 1}}$$
where $\gamma_x:S^x\cap\prescript{u_i}{}{H}/S^x\cap\prescript{u_i}{}{\mathbf 1}\stackrel{\sim}{\longrightarrow}(S^x\cap\prescript{u_i}{}{H})\prescript{u_i}{}{\mathbf 1}/\prescript{u_i}{}{\mathbf 1}$
is the canonical isomorphism.
Next we have
$$\mathrm{Res}^{\prescript{u_i}{}{H}/\prescript{u_i}{}{\mathbf 1}}_{(S^x\cap\prescript{u_i}{}{H})\prescript{u_i}{}{H}/\prescript{u_i}{}{\mathbf 1}}\circ\mathrm{Iso}(\bar\varphi_{u_i}^{-1})\cong\mathrm{Iso}((\bar\varphi_{xu_i}\big|_{(S^x\cap\prescript{u_i}{}{H})\prescript{u_i}{}{\mathbf 1}/\prescript{u_i}{}{\mathbf 1}})^{-1})\circ\mathrm{Res}^{G^{u_i}}_{S^{xu_i}}$$
where the isomorphism $(S^x\cap\prescript{u_i}{}{H})\prescript{u_i}{}{\mathbf 1}/\prescript{u_i}{}{\mathbf 1}\stackrel{\sim}{\longrightarrow}S^{xu_i}$
is given by restricting $\bar\varphi_{u_i}$ to
$(S^x\cap\prescript{u_i}{}{H})\prescript{u_i}{}{\mathbf 1}/\prescript{u_i}{}{\mathbf 1}$.
Finally, we have
$\mathrm{Res}^{G^{u_i}}_{S^{xu_i}}\circ\mathrm{Res}^H_{G^{u_i}}\cong\mathrm{Res}^H_{S^{xu_i}}$
by transitivity.
Then putting these together, we see that
$$V\circ U_i\cong\coprod_{x\in S\backslash G/\prescript{u_i}{}{H}}V_{i,x}\circ\mathrm{Res}^H_{S^{xu_i}}$$
where $V_{i,x}$ is some right-free $(S,S^{xu_i})$-biset.
So finally, we have
$$V\circ U\cong\coprod_{i=1}^t\coprod_{x\in S\backslash G/\prescript{u_i}{}{H}}V_{i,x}\circ\mathrm{Res}^H_{S^{xu_i}}.$$
Then by Proposition \ref{tUprops}, we see that
$$\mathrm{Res}_S^G(t_U(M))\cong\bigotimes_{i=1}^t
\bigotimes_{x\in S\backslash G/\prescript{u_i}{}{H}}t_{V_{i,x}}(\mathrm{Res}^H_{S^{xu_i}}(M)),$$
with a similar result for $N$ in place of $M$.
Since $S$ is a $p'$-subgroup of $G$, so is $S\cap\prescript{xu_i}{}{H}$
for all $x\in G$ and $i=1,\dots,t$.
Hence $S^{xu_i}=\varphi_{xu_i}(S\cap\prescript{xu_i}{}{H})$ is a $p'$-subgroup of $H$.
So by the assumption, $\mathrm{Res}_{S^{xu_i}}^H(M)\cong\mathrm{Res}_{S^{xu_i}}^H(N)$
for all $x\in G$ and $i=1,\dots,t$.
Then by functoriality,
$t_{V_{i,x}}(\mathrm{Res}_{S^{xu_i}}^H(M))\cong t_{V_{i,x}}(\mathrm{Res}_{S^{xu_i}}^H(N))$ for all $x\in G$ and $i=1,\dots,t$.
From here, we can conclude that $\mathrm{Res}^G_S(t_U(M))\cong\mathrm{Res}^G_S(t_U(N))$,
showing that $R_F(U)$ is well-defined.
Is it then clear that $R_F(U)$ is multiplicative and is algebraic of degree $|U/H|$.
Hence we have a multiplicative, dual-preserving function
$R_F(U):R_F(H)\longrightarrow R_F(G)$ which again is algebraic of degree $|U/H|$.

From here we want to show that $R_F(U)\circ d_H=d_G\circ R_K(U)$.
To do so, we just need to show that both compositions agree on $R_K^+(H)$.
So let $M$ be a $KH$-module,
and let $L\subseteq M$ be a full $\mathcal OH$-lattice in $M$
so that $d_H([M])=[F\otimes_{\mathcal O}L]$.
Next we see that $t_U(L)$ is a full $\mathcal OG$-lattice inside $t_U(M)$,
and therefore $$(d_G\circ R_K(U))([M])=d_G([t_U(M)])=[F\otimes_{\mathcal O}t_U(L)].$$
Before proceeding, we prove the following lemma:

\begin{lemma}\label{changescalars}
Let $k,k'$ be two commutative rings with a morphism $k\longrightarrow k'$ of rings.
Then the functors $t_U(k\otimes_{k'}-),k\otimes_{k'}t_U(-):\Mod{k'H}\longrightarrow\Mod{kG}$
are naturally isomorphic.
\end{lemma}
\begin{proof}
If $M$ is a $k'H$-module, then the map $t_U(k\otimes_{k'}M)\longrightarrow k\otimes_{k'}t_U(M)$
defined by $$(c_1\otimes m_1)\otimes\cdots\otimes(c_n\otimes m_n)\mapsto
\left(\prod_{i=1}^nc_i\right)\otimes(m_1\otimes\cdots\otimes m_n)$$
is clearly an isomorphism of $kG$-modules, natural in $M$.
\end{proof}

Then applying this lemma to the canonical epimorphism $\mathcal O\longrightarrow F$,
we have the following:
\begin{align*}
(R_F(U)\circ d_H)([M])
&=R_F(U)(F\otimes_{\mathcal O}L)
=[t_U(F\otimes_{\mathcal O}L)]\\
&=[F\otimes_{\mathcal O}t_U(L)]
=d_G([t_U(M)])\\
&=(d_G\circ R_K(U))([M])
\end{align*}

\noindent From here we can conclude that $R_F(U)\circ d_H=d_G\circ R_K(U)$.
Recall that the decomposition $d_H$ is always surjective.
In particular, if $M$ is an $FH$-module with Brauer character $\psi$,
then we have its Brauer lift $\hat\psi\in R_K(H)$ 
where $\hat\psi(h)=\psi(h_{p'})$ for all $h\in H$.
We can then compute $R_K(\hat\psi)$ using Proposition \ref{inducedchar}.
Then after applying $d_G$, the following theorem becomes clear:

\begin{theorem}\label{inducedBchar}
If $M$ is an $FH$-module with Brauer character $\psi$,
then $t_U(M)$ is an $FG$-module with Brauer character $\psi^U$,
where for $x\in\mathscr E_p(G),$
$$\psi^U(x)=\prod_{u\in\langle x\rangle\backslash U/H}\psi(\varphi_u(x^{n_u})),
\quad n_u=[\langle x\rangle:\langle x\rangle\cap\prescript{u}{}{H}].$$
\end{theorem}

To conclude this section, we define $\tilde R_F(U)$ to complete the rest of the diagram:
$$\tilde R_F(U):\tilde R_F(H)\longrightarrow\tilde R_F(G),\quad
(w_h)_{h\in\mathscr E_p(H)}\mapsto
\left(\prod_{u\in\langle x\rangle\backslash U/H}w_{\varphi_u(x^{n_u})}\right)_{x\in\mathscr E_p(G)}.$$
It is clear that $\tilde R_F(U)$ is multiplicative, preserves duals, is algebraic of degree $|U/H|$,
and extends the function $R_F(U)$ to ghost rings.
Again since $\xi_H$ is injective with finite cokernel,
we can see that $\tilde R_F(U)$ is the unique multiplicative extension of $R_F(U)$.
Finally we have the following diagram where each face is commutative.
We omit the names of the maps for simplicity.
Each function is multiplicative and preserves duals.
The dashed arrows are algebraic of degree $|U/H|$.
The solid arrows are all additive (hence algebraic of degree 1).

\begin{equation}
\label{bigpic}
\raisebox{-0.5\height}{\includegraphics{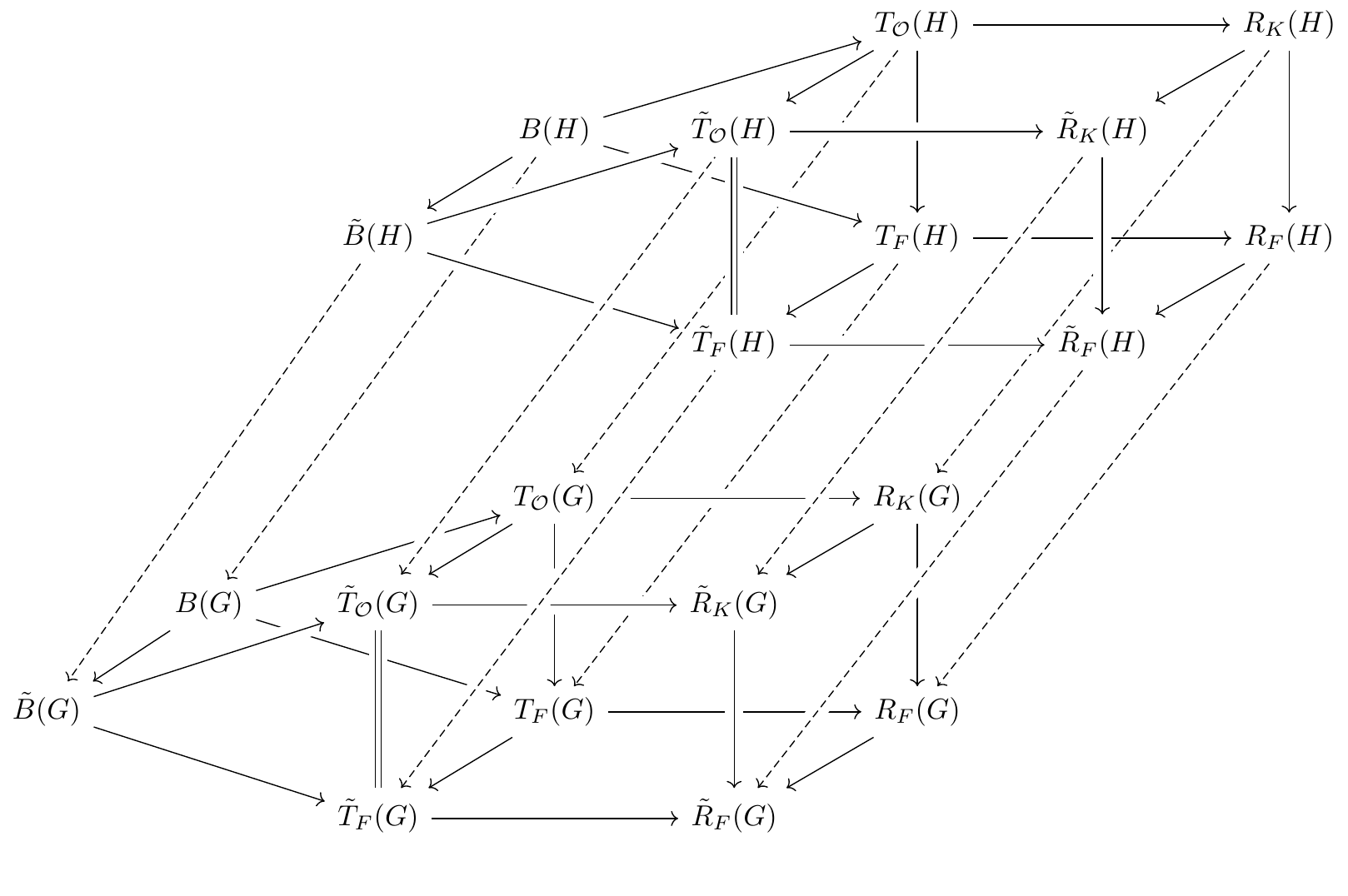}}
\end{equation}

\section{Inflation Functors}\label{InfFun}
Here we recall the notion of inflation functors (see \cite{Bouc2} for more details).
We let $B(G,H)$ denote the Grothendieck group of the isomorphism classes of finite $(G,H)$-bisets
with respect to disjoint unions.
In other words, $B(G,H)$ is just $B(G\times H^{\text{op}})$.
Next we let $I(G,H)$ denote the subgroup of $B(G,H)$
generated by the isomorphism classes of finite right-free $(G,H)$-bisets.
Every element of $I(G,H)$ can be written in the form $[U]-[U']$
for some right-free $(G,H)$-bisets $U$ and $U'$ (though not necessarily in a unique way).
If $K$ is an additional finite group,
then there is a bilinear pairing

$$I(G,H)\times I(H,K)\longrightarrow I(G,K),\quad(u,v)\mapsto u\times_Hv$$

\noindent induced by tensoring over $H$.
We can then define the category $\mathcal{I}$ to have as objects all finite groups
and $\text{Hom}_{\mathcal{I}}(H,G)=I(G,H)$.
For $u\in\text{Hom}_{\mathcal{I}}(H,G)$ and $v\in\text{Hom}_{\mathcal{I}}(K,H)$,
the composition $u\circ v$ is defined to be $u\times_Hv$.
The identity in $\text{Hom}_{\mathcal{I}}(G,G)$ is $[\text{Id}_G]$,
where $\text{Id}_G$ is $G$ as a $(G,G)$-biset with multiplication from $G$ on both sides.
Since the morphism sets of $\mathcal I$ are abelian groups and composition is bilinear,
we see that $\mathcal I$ is a preadditive category.
An inflation functor is then an additive functor
$\mathcal I\longrightarrow\mathsf{Ab}$,
where $\mathsf{Ab}$ denotes the category of abelian groups.

We will write all abelian groups multiplicatively.
So if $A$ and $B$ are two abelian groups,
the ``zero'' morphism of $\text{Hom}_{\mathsf{Ab}}(A,B)$
is the homomorphism $a\mapsto1_B$ for all $a\in A$, where $1_B$ is the identity of $B$.
If $f,g\in\text{Hom}_{\mathsf{Ab}}(A,B)$, then $fg\in\text{Hom}_{\mathsf{Ab}}(A,B)$
is the homomorphism $a\mapsto f(a)g(a)$.
And $f^{-1}$, the inverse of $f$ in $\text{Hom}_{\mathsf{Ab}}(A,B)$
is the homomorphism $a\mapsto f(a)^{-1}$.

\section{Unit Groups as Inflation Functors}\label{UnitGp}
The goal of this section is to define inflation functors
for each of the unit groups of the representation rings and their ghost rings discussed above.
We let $K=\bar\Q_p$, the algebraic closure of the $p$-adic numbers
and $\mathcal O$ be the integral closure of $\Z_p$ in $K$.
Although $\mathcal O$ is not a complete DVR, it has a unique maximal ideal $\mathfrak p$.
We then set $F=\mathcal O/\mathfrak p,$ which is an algebraically closed field of characteristic $p$.
Now for any fixed finite group $G$,
the canonical map $T_{\mathcal O}(G)\rightarrow T_F(G)$ is still an isomorphism.
We can then find a $p$-modular system $(K',\mathcal O', F')$ with embeddings
$K'\hookrightarrow K,\mathcal O'\hookrightarrow\mathcal O,$ and $F'\hookrightarrow F$
that give isomorphisms
$R_{K'}(G)\stackrel{\sim}{\rightarrow}R_K(G),$
$R_{F'}(G)\stackrel{\sim}{\rightarrow}R_F(G),$
$T_{F'}(G)\stackrel{\sim}{\rightarrow}T_F(G),$
and $T_{\mathcal O'}(G)\stackrel{\sim}{\rightarrow}T_{\mathcal O}(G)$.
So with some abuse of notation,
when focusing on a particular group $G$,
we will identify $T_{\mathcal O}(G)$ with $T_{\mathcal O'}(G)$
(with similar identifications for the other representation rings).

For each finite group $G$, we have a ring $R(G)$ where $R(-)$
will stand for $B(-),T_{\mathcal O}(-), T_F(-), R_F(-), R_K(-)$,
or one of their associated ghost rings.
And if $G$ and $H$ are two finite groups and $U$ is a right-free $(G,H)$-biset,
we have a multiplicative function $R(U):R(H)\longrightarrow R(G)$,
hence a group homomorphism $R(U)^\times:R(H)^\times\longrightarrow R(G)^\times$.
We will show that each possible $R(-)^\times$ defines an inflation functor.
Now if $a\in I(G,H)$, then $a=[U]-[U']$ for some right-free $(G,H)$-bisets $U$ and $U'$.
Hence, we can define the group homomorphism
$R(a)^\times:=R(U)^\times(R(U')^\times)^{-1}:R(H)^\times\longrightarrow R(G)^\times$.
We will show explicitly for $R(-)=\tilde T_F(-)$
that this makes $\tilde T_F(-)^\times$ an inflation functor,
and then it will be clear that similar proofs will show $R(-)^\times$ is also an inflation functor
for the other choices of $R(-)$.

To see that $\tilde T_F(a)$ is well-defined,
suppose additionally that $a=[X]-[X']$ for some other right-free $(G,H)$-bisets $X$ and $X'$.
Then $[U]-[U']=[X]-[X']$, hence $[U\sqcup X']=[X\sqcup U']$ in $I(G,H)$.
Therefore $U\sqcup X'$ and $X\sqcup U'$ are isomorphic right-free $(G,H)$-bisets.
So $\tilde T_F(U\sqcup X')^\times=\tilde T_F(X\sqcup U')^\times$.
Now if $(E,c)\in\mathscr T_p(G)$,
$(u_1,\dots,u_n)$ is a set of representatives of $E\backslash U/H$,
and $(x_1',\dots,x_m')$ is a set of representatives of $E\backslash X'/H$,
then clearly $(u_1,\dots,u_n,x_1',\dots,x_m')$ is a set of representatives
of $E\backslash(U\sqcup X')/H$.
From here we can see that
$\tilde T_F(U\sqcup X')^\times=\tilde T_F(U)^\times\tilde T_F(X')^\times$.
Similarly, $\tilde T_F(X\sqcup U')^\times=\tilde T_F(X)^\times\tilde T_F(U')^\times$.
So we see that $\tilde T_F(U)^\times\tilde T_F(X')^\times=\tilde T_F(X)^\times\tilde T_F(U')^\times$.
Hence $\tilde T_F(U)^\times(\tilde T_F(U')^\times)^{-1}=
\tilde T_F(X)^\times(\tilde T_F(X')^\times)^{-1}$.
Thus we see that $\tilde T_F(a)^\times$ is well-defined.
So for each finite group $G$, we have an abelian group $\tilde T_F(G)^\times$,
and for every $a\in I(G,H)$, we have a homomorphism of abelian groups
$\tilde T_F(a)^\times:\tilde T_F(H)^\times\longrightarrow\tilde T_F(G)^\times$.

It is clear that $\tilde T_F([\text{Id}_G])^\times:\tilde T_F(G)^\times\longrightarrow\tilde T_F(G)^\times$
is the identity function.
So we next show that $\tilde T_F(-)^\times$ preserves compositions.
Suppose that $G,H,$ and $K$ are finite groups,
$U$ is a right-free $(G,H)$-biset, and $V$ is a right-free $(H,K)$-biset.
Then $U\times_HV$ is a right-free $(G,K)$-biset.
We will show that $\tilde T_F(U\times_HV)^\times=\tilde T_F(U)^\times\circ\tilde T_F(V)^\times$.
First note that if $u\in U,v\in V$, and $T\leq K,$
then $\prescript{(u,_Hv)}{}{T}=\prescript{u}{}{(\prescript{v}{}{T})}$.
Similarly $S^{(u,_Hv)}=(S^u)^v$ for any $S\leq G$.
In particular, for $(E,c)\in\mathscr T_p(G)$,
we have $E^{(u,_Hv)}=(E^u)^v$ and $O_p(E)^{(u,_Hv)}=(O_p(E)^u)^v$.
Hence $c^{(u,_Hv)},(c^u)^v\in E^{(u,_Hv)}/O_p(E)^{(u,_Hv)}$.
We will show that in fact, $c^{(u,_Hv)}=(c^u)^v$.
We have $e_{(u,_Hv)}=[E:O_p(E)(E\cap\prescript{(u,_Hv)}{}{K})]$
and $f_{(u,_Hv)}=[O_p(E)(E\cap\prescript{(u,_Hv)}{}{K}):E\cap\prescript{(u,_Hv)}{}{K}]$.
Hence $e_{(u,_Hv)}f_{(u,_Hv)}=[E:E\cap\prescript{(u,_Hv)}{}{K}]$.
Next we have $e_u=[E:O_p(E)(E\cap\prescript{u}{}{H})]$
and $f_u=[O_p(E)(E\cap\prescript{u}{}{H}):E\cap\prescript{u}{}{H}].$
Hence $e_uf_u=[E:E\cap\prescript{u}{}{H}]$.
Now $\prescript{(u,_Hv)}{}{K}=\prescript{u}{}{(\prescript{v}{}{K})}\leq\prescript{u}{}{H}$.
Hence $E\cap\prescript{u}{}{(\prescript{v}{}{K})}\leq E\cap\prescript{u}{}{H}$, and we have
$$e_{(u,_Hv)}f_{(u,_Hv)}=[E:E\cap\prescript{u}{}{(\prescript{v}{}{K})}]
=[E:E\cap\prescript{u}{}{H}][E\cap\prescript{u}{}{H}:E\cap\prescript{u}{}{(\prescript{v}{}{K})}]
=e_uf_u[E\cap\prescript{u}{}{H}:E\cap\prescript{u}{}{(\prescript{v}{}{K})}].$$
We see that $\varphi_u$ maps $E\cap\prescript{u}{}{H}$ onto $E^u$
and $E\cap\prescript{u}{}{(\prescript{v}{}{K})}$ onto $E^u\cap\prescript{v}{}{K}$.
Hence $[E^u:E^u\cap\prescript{v}{}{K}]=
[E\cap\prescript{u}{}{H}:(E\cap\prescript{u}{}{(\prescript{v}{}{K})})(E\cap\text{Ker}(\varphi_u))]$
by an isomorphism theorem.
But $E\cap\text{Ker}(\varphi_u)\leq E\cap\prescript{u}{}{(\prescript{v}{}{K})}.$
So we have $[E^u:E^u\cap\prescript{v}{}{K}]=
[E\cap\prescript{u}{}{H}:E\cap\prescript{u}{}{(\prescript{v}{}{K})}].$
Now setting $e_v=[E^u:O_p(E)^u(E^u\cap\prescript{v}{}{K})]$
and $f_v=[O_p(E)^u(E^u\cap\prescript{v}{}{K}):E^u\cap\prescript{v}{}{K}],$
we then have $e_{(u,_Hv)}f_{(u,_Hv)}=e_uf_ue_vf_v.$

Now suppose $c=sO_p(E)$ for some $s\in E$.
Then $s^{e_{(u,_Hv)}}=ax$ for some $a\in O_p(E)$ and $x\in E\cap\prescript{(u,_Hv)}{}{K}$.
Then $c^{(u,_Hv)}=\varphi_{(u,_Hv)}(x)^{f_{(u,_Hv)}}O_p(E)^{(u,_Hv)}$.
On the other hand, $s^{e_u}=by$ for some $b\in O_p(E)$ and $y\in E\cap\prescript{u}{}{H}$.
So $c^u=\varphi(y)^{f_u}O_p(E)^u$.
Then $\varphi_u(y)^{f_ue_v}=tz$ for some $t\in O_p(E)^u$ and $z\in E^u\cap\prescript{v}{}{K}$.
Therefore $(c^u)^v=\varphi_v(z)^{f_v}O_p(E)^{(u,_Hv)}$.
So we will show that
$\varphi_{(u,_Hv)}(z)^{f_{(u,_Hv)}}O_p(E)^{(u,_Hv)}=\varphi_v(z)^{f_v}O_p(E)^{(u,_Hv)}$.

Now $s^{e_{(u,_Hv)}f_{(u,_Hv)}}=(ax)^{f_{(u,_Hv)}}=a'x^{f_{(u,_Hv)}}$ for some $a'\in O_p(E)$.
Hence we have $s^{e_{(u,_Hv)}f_{(u,_Hv)}}(u,_Hv)=a'(u,_Hv)\varphi_{(u,_Hv)}(x)^{f_{(u,_Hv)}}$.
On the other hand,
$s^{e_{(u,_Hv)}f_{(u,_Hv)}}=s^{e_uf_ue_vf_v}=(by)^{f_ue_vf_v}=b'y^{f_ue_vf_v}$
for some $b'\in O_p(E)$.
Hence also $$s^{e_{(u,_Hv)}f_{(u,_Hv)}}(u,_Hv)=b'(y^{f_ue_vf_v}u,_Hv)
=b'(u\varphi_u(y)^{f_ue_vf_v},_Hv)=b'(u,_H\varphi_u(y)^{f_ue_vf_v}v)
=b'(u,_H(cz)^{f_v}v).$$
Now $(tz)^{f_v}=t'z^{f_v}$ for some $t'\in O_p(E)^u$.
Hence $t'=\varphi_u(q)$ for some $q\in O_p(E)$.
Thus $$b'(u,_H(cz)^{f_v}v)=b'(u,_H\varphi_u(q)z^{f_v}v)=b'(u\varphi_u(q),_Hv\varphi_v(z)^{f_v})
=b'(qu,_Hv)\varphi_v(z)^{f_v}=b'q(u,_Hv)\varphi_v(z)^{f_v}.$$
So altogether we have $a'(u,_Hv)\varphi_{(u,_Hv)}(x)^{f_{(u,_Hv)}}=b'q(u,_Hv)\varphi_v(z)^{f_v}$.
Since $a',b',q\in O_p(E)$, this implies that
$\varphi_{(u,_Hv)}(x)^{f_{(u,_Hv)}}O_p(E)^{(u,_Hv)}=\varphi_v(z)^{f_v}O_p(E)^{(u,_Hv)}$.
Thus $c^{(u,_Hv)}=(c^u)^v$ as claimed.

Next we claim that $E\backslash U\times_HV/K
\cong\coprod_{u\in E\backslash U/H}E^u\backslash V/K$.
So pick a set of representative $E\backslash U/H=\{u_1,\dots,u_k\}$.
And for $j=1,\dots,k$, pick a set of representatives
$E^{u_j}\backslash V/K=\{v_{j1},\dots,v_{jl_j}\}$.
We will show that $\{(e_j,_Hv_{ji}):j=1,\dots,k,i=1,\dots,l_j\}$
is a set of representatives of $E\backslash U\times_HV/K$.
So let $(u,_Hv)\in U\times_HV$.
Then $u=au_jh$ for some $a\in E,j\in\{1,\dots,k\}$ and $h\in H$.
And $hv=a_jv_{ji}k$ for some $a_j\in E^{u_j},i\in\{1,\dots,l_j\},$ and $k\in K$.
Since $a_j\in E^{u_j}$, we have $a_j=\varphi_{u_j}(q)$ for some $q\in E$.
Thus $$(u,_Hv)=(au_jh,_Hv)=a(u_j,_Hhv)=a(u_j,_Ha_jv_{ji}k)
=a(u_ja_j,_Hv_{ji})k=a(qu_j,_Hv_{ji})k=aq(u_j,_Hv_{ji})k,$$
showing $(u,_Hv)$ is in the same $(E,K)$-class of $(u_j,_Hv_{ji})$.
So every element of $U\times_HV$ is in the same $(E,K)$-class of some $(u_j,_Hv_{ji})$.
Now to see these classes are all distinct,
suppose that $(u_j,_Hv_{ji})=a(u_{j'},_Hv_{j'i'})k$ for some $a\in E$ and $k\in K$.
Then $(u_j,_Hv_{ji})=(au_{j'},_Hv_{j'i'}k)$ implying there exists $h\in H$
such that $u_j=au_{j'}h$ and $v_{ji}=h^{-1}v_{j'i'}k$.
Since $u_j=au_{j'}h,$ we see that $j=j'$ and $h\in E^{u_j}$.
Then $h^{-1}\in E^{u_j}$ as well, and $v_{ji}=h^{-1}v_{ji'}k$ implies $i=i'$.
So we see all these classes are distinct.
Thus $\{(u_j,_Hv_{ji}):j=1,\dots,k,i=1,\dots,l_j\}$ is a complete set of representatives
of $E\backslash U\times_HV/K$.
Hence $E\backslash U\times_HV/K
\cong\coprod_{u\in E\backslash U/H}E^u\backslash V/K$ as claimed.

Now to put this altogether, we have the following:
\begin{align*}
(\tilde T_F(U)^\times\circ\tilde T_F(V)^\times)((z_{(C,a)})_{(C,a)\in\mathscr T_p(K)})
&=\tilde T_F(U)^\times\left(\left(\prod_{v\in D\backslash V/K}z_{(D^v,b^v)}\right)_{(D,b)\in\mathscr T_p(H)}\right)\\
&=\left(\prod_{u\in E\backslash U/H}\prod_{v\in E^u\backslash V/K}z_{((E^u)^v,(c^u)^v)}\right)_{(E,c)\in\mathscr T_p(G)}\\
&=\left(\prod_{(u,_Hv)\in E\backslash U\times_HV/K}z_{(E^{(u,_Hv)},c^{(u,_Hv)})}\right)_{(E,c)\in\mathscr T_p(G)}\\
&=\tilde T_F(U\times_HV)^\times((z_{(C,a)})_{(C,a)\in\mathscr T_p(K)})
\end{align*}
So we see that
$\tilde T_F(U)^\times\circ\tilde T_F(V)^\times=\tilde T_F(U\times_HV)^\times.$
Now suppose that $a\in I(G,H)$ and $b\in I(H,K)$.
Then $a=[U]-[V]$ for some right-free $(G,H)$-bisets $U$ and $V$,
and $b=[Z]-[W]$ for some right-free $(H,K)$-bisets $Z$ and $W$.
Then $a\circ b=[U\times_HZ\sqcup V\times_HW]-[U\times_HW\sqcup V\times_HW]$,
and we see that
\begin{align*}
\tilde T_F(a\circ b)^\times
&=\tilde T_F(U\times_HZ\sqcup V\times_HW)^\times
(\tilde T_F(U\times_HW\sqcup V\times_HZ)^\times)^{-1}\\
&=(\tilde T_F(U)^\times\circ\tilde T_F(Z)^\times)
(\tilde T_F(V)^\times\circ\tilde T_F(W)^\times)
(\tilde T_F(U)^\times\circ\tilde T_F(W)^\times)^{-1}
(\tilde T_F(V)^\times\circ\tilde T_F(Z)^\times)^{-1}\\
&=(\tilde T_F(U)^\times\circ\tilde T_F(Z)^\times)
((\tilde T_F(V)^\times)^{-1}\circ(\tilde T_F(W)^\times)^{-1})
(\tilde T_F(U)^\times\circ(\tilde T_F(W)^\times)^{-1})
((\tilde T_F(V)^\times)^{-1}\circ\tilde T_F(Z)^\times)\\
&=(\tilde T_F(U)^\times\circ\tilde T_F(Z)^\times(\tilde T_F(W)^\times)^{-1})
((\tilde T_F(V)^\times)^{-1}\circ\tilde T_F(Z)^\times(\tilde T_F(W)^\times)^{-1})\\
&=\tilde T_F(U)^\times(\tilde T_F(V)^\times)^{-1}\circ
\tilde T_F(Z)^\times(\tilde T_F(W)^\times)^{-1}\\
&=\tilde T_F(a)^\times\circ\tilde T_F(b)^\times
\end{align*}

This shows that $\tilde T_F(-)^\times:\mathscr I\longrightarrow\mathsf{Ab}$ is a functor.
Lastly, we must show that $\tilde T_F(-)^\times$ is additive.
So suppose $a,b\in I(G,H)$.
Then $a=[U]-[U']$ and $b=[V]-[V']$ for some right-free $(G,H)$-bisets $U,U',V,$ and $V'$.
So $a+b=[U\sqcup V]-[U'\sqcup V'],$
and we have the following:
\begin{align*}
\tilde T_F(a+b)^\times
&=\tilde T_F(U\sqcup V)^\times(\tilde T_F(U'\sqcup V')^\times)^{-1}\\
&=\tilde T_F(U)^\times\tilde T_F(V)^\times(\tilde T_F(U')^\times\tilde T_F(V')^\times)^{-1}\\
&=\tilde T_F(U)^\times(\tilde T_F(U')^\times)^{-1}\tilde T_F(V)^\times(\tilde T_F(V')^\times)^{-1}\\
&=\tilde T_F(a)^\times\tilde T_F(b)^\times
\end{align*}

So we see that $\tilde T_F(-)^\times:\mathscr I\longrightarrow\mathsf{Ab}$ is additive,
hence an inflation functor.
From here the corresponding result
for the unit groups of the other representation rings and their ghost rings should be clear.
Returning to Diagram \ref{bigpic},
we recall that all arrows are multiplicative,
hence restricting to unit groups gives a commutative diagram in $\mathsf{Ab}$.
So all the connecting maps induce morphisms of inflation functors.
In particular, all the ghost maps define morphisms of inflation functors.

\section{Orthogonal Units}\label{OrthUnits}
In this last section, we introduce the idea of \emph{orthogonal units}
in the various representation rings and their ghost rings.
Recall that every ring of interest has a duality operator on the ring.
We define an orthogonal unit to be a unit whose inverse is its dual element.
The set of all orthogonal units forms a subgroup of the unit group in each case.
We first like to show that in each case the orthogonal unit group
is just the torsion subgroup of the full unit group.

Again we focus on the trivial source ring and note the corresponding result for the other rings.
So we denote the orthogonal unit group
$U_\circ(T_F(G)):=\{a\in T_F(G)^\times:a^{-1}=a^\circ\}\leq T_F(G)^\times$.
Similarly, $U_\circ(\tilde T_F(G)):=\{b\in\tilde T_F(G)^\times:b^{-1}=b^\circ\}\leq\tilde T_F(G)^\times$.
Now $b=(z_{(E,c)})_{(E,c)\in\mathscr T_p(G)}\in\tilde T_F(G)^\times$
iff each $z_{(E,c)}$ is a unit of $\Z[\mu]$.
Moreover, $b\in U_\circ(\tilde T_F(G))$ iff each $z_{(E,c)}$ is an orthogonal unit of $\Z[\mu]$.
That is $z_{(E,c)}^{-1}=\sigma_{-1}(z_{(E,c)})$, and therefore $z_{(E,c)}\sigma_{-1}(z_{(E,c)})=1$.
Now Theorem 4.12 in \cite{Washington} implies
that in this situation $z_{(E,c)}$ is a root of unity in $\mathcal O$,
and if $e=\text{exp}(G)_{p'}$, we have $U_\circ(\tilde T_F(G))=
(\prod_{(E,c)\in\mathscr T_p(G)}\{\pm\mu^f:f=0,\dots,e-1\})^{G\times\Delta}.$
This is then just the torsion subgroup of $\tilde T_F(G)^\times$.
We want to also show that $U_\circ(T_F(G))$ is the torsion subgroup of $T_F(G)^\times$.
If $a\in U_\circ(T_F(G))$, then since $\tau_G$ is multiplicative and preserves duals,
$\tau_G(a)\in U_\circ(\tilde T_F(G))$.
Hence $\tau_G(a)$ is a torsion element of $\tilde T_F(G)$, thus has finite order.
Then since $\tau_G$ is injective,
this implies that also $a$ has finite order in $T_F(G)^\times$.
Hence $U_\circ(T_F(G))$ is a torsion subgroup of $T_F(G)^\times$.
But if $a\in T_F(G)^\times$ has finite order $n$,
then $\tau_G(a)$ has order $n$ in $\tilde T_F(G)^\times$ since $\tau_G$ is injective.
So $\tau_G(a)\in U_\circ(\tilde T_F(G))$ and therefore
$\tau_G(a^{-1})=\tau_G(a)^{-1}=\tau_G(a)^\circ=\tau_G(a^\circ)$,
since also $\tau_G$ preserves duals.
Again since $\tau_G$ is injective, we see that $a^{-1}=a^\circ$.
So every torsion unit of $T_F(G)$ is orthogonal.
Thus we conclude that $U_\circ(T_F(G))$ is precisely the torsion subgroup of $T_F(G)^\times$.

We can similarly define orthogonal unit groups for the other representation rings and their ghost rings.
In all cases, we see that the group of orthogonal units is just the torsion subgroup of the full unit group.
So restricting further to the torsion subgroup of the group of units,
we get sub-inflation functors for all representation and ghost rings.
Notice that $U_\circ(B(G))=B(G)^\times$ since every element of $B(G)^\times$ has finite order,
in fact has order dividing 2.
The theory of biset functors has been instrumental in studying $B(G)^\times.$
The hope is that can we similarly study $U_\circ(T_F(G))$ using this inflation functor theory.
Now also, we have $U_\circ(R_K(G))=\{\chi\in R_K(G)^\times:\chi^{-1}=\chi^\circ\}$,
and similarly $U_\circ(R_F(G))$.
These two orthogonal unit groups can be determined quite easily,
and we finish with these results.

We let $\hat G$ denote the group of linear characters of $G$.
That is, $\hat G$ is the set of homomorphisms $G\longrightarrow\mathcal O^\times$.
This is a group under pointwise multiplication.
In \cite{Yamauchi}, Kenichi Yamauchi proved that $U_\circ(R_K(G))=\{\pm\chi:\chi\in\hat G\}$.
Hence $U_\circ(R_K(G))$ is a finite group of order $2[G:G']$,
where $G'$ denotes the commutator subgroup of $G$. 
Now to determine, $U_\circ(R_F(G))$ recall that $d_G$ is surjective and has a section
$m_G:R_F(G)\longrightarrow R_K(G)$ defined by $m_G(\psi)(x)=\psi(x_{p'})$ for all $x\in G$.
We see in \cite{CurRein} that $m_G$ is a ring morphism that preserves duals.
Hence $m_G(U_\circ(R_F(G)))\leq U_\circ(R_K(G))$.
On the other hand, since $d_G$ is a ring morphism preserving duals,
we also have $U_\circ(R_F(G))=d_G(m_G(U_\circ(R_F(G))))
\leq d_G(U_\circ(R_K(G)))\leq U_\circ(R_F(G))$.
So we have $d_G(U_\circ(R_K(G)))=U_\circ(R_F(G))$.
Hence every element of $U_\circ(R_F(G))$ is just the restriction of an element of $U_\circ(R_K(G))$
to the set of $p'$-elements of $G$.
We have thus proven the following theorem:

\begin{theorem}\label{OrthBra}
$U_\circ(R_F(G))=\{\pm\chi|_{\mathscr E_p(G)}:\chi\in\hat G\}$.
\end{theorem}


\end{document}